\documentclass[reqno,12pt]{amsart}
\usepackage{a4wide,color,eucal,enumerate,mathrsfs}
\usepackage[normalem]{ulem}
\usepackage{amsmath,amssymb,epsfig,amsthm} 
\usepackage[latin1]{inputenc}
\usepackage{graphicx}
\usepackage{psfrag}
\usepackage{hyperref,cleveref,tikz}
\usepackage{pict2e}

\def\inter{{\operatorname{int}}}

\def\vint{\mathop{\mathchoice%
		{\setbox0\hbox{$\displaystyle\intop$}\kern 0.22\wd0%
			\vcenter{\hrule width 0.6\wd0}\kern -0.82\wd0}%
		{\setbox0\hbox{$\textstyle\intop$}\kern 0.2\wd0%
			\vcenter{\hrule width 0.6\wd0}\kern -0.8\wd0}%
		{\setbox0\hbox{$\scriptstyle\intop$}\kern 0.2\wd0%
			\vcenter{\hrule width 0.6\wd0}\kern -0.8\wd0}%
		{\setbox0\hbox{$\scriptscriptstyle\intop$}\kern 0.2\wd0%
			\vcenter{\hrule width 0.6\wd0}\kern -0.8\wd0}}%
	\mathopen{}\int}

\def\az{\alpha}

\def\dist{{\mathop\mathrm{\,dist\,}}}

\def\dz{\delta}

\def\ez{\epsilon}

\def\bz{\beta}

\newcommand{\average}{{\mathchoice {\kern1ex\vcenter{\hrule height.4pt
width 6pt depth0pt} \kern-9.7pt} {\kern1ex\vcenter{\hrule
height.4pt width 4.3pt depth0pt} \kern-7pt} {} {} }}

\def\bint{{\ifinner\rlap{\bf\kern.35em--}
\int\else\rlap{\bf\kern.45em--}\int\fi}\ignorespaces}

\def\bbint{{\ifinner\rlap{\bf\kern.35em--}
\hspace{0.078cm}\int\else\rlap{\bf\kern.45em--}\int\fi}\ignorespaces}

\def\diam{{\mathop\mathrm{\,diam\,}}}

\def\diver{{\mathop{\mathrm{div}}}}

\newcommand{\R}{\mathbb{R}}

\newtheorem{thm}{Theorem}[section]
\newtheorem{lem}[thm]{Lemma}
\newtheorem{prop}[thm]{Proposition}
\newtheorem{cor}[thm]{Corollary}
\newtheorem{defn}[thm]{Definition}
\numberwithin{equation}{section}

\theoremstyle{remark}
\newtheorem{rem}[thm]{Remark}

\def\bint{{\ifinner\rlap{\bf\kern.35em--}
\int\else\rlap{\bf\kern.45em--}\int\fi}\ignorespaces}

\usepackage{graphicx}
\usepackage{import}
\usepackage{xifthen}
\usepackage{pdfpages}
\usepackage{transparent}

\title{Geometric properties of Euclidean domains supporting trace inequalities}
\author{Weicong Su, Zhuang Wang, Yi Ru-Ya Zhang }
\date{\today}

\address{State Key Laboratory of Mathematical Sciences, Academy of Mathematics and Systems Science, Chinese Academy of Sciences, Beijing 100190, China}
\address{Institute of Mathematics, Academy of Mathematics and Systems Science, the Chinese Academy of Sciences, Beijing 100190, China}
\email{suweicong@amss.ac.cn}
\email{yzhang@amss.ac.cn}
\address{Key Laboratory of Computing and Stochastic Mathematics (Ministry of Education), School of Mathematics and Statistics, Hunan Normal University, Changsha, Hunan 410081, China.}
\email{zwang@hunnu.edu.cn}

 \thanks{The first and the third authors are funded by the National Key R\&D Program of China (Grant No. 2025YFA1018400 \&  No. 2021YFA1003100), NSFC Grant No. 12288201 \& No. 12571128, the Chinese Academy of Sciences, and CAS Project for Young Scientists in Basic Research, Grant No. YSBR-031. The second author is funded by the Natural Science Foundation of Hunan Province (Grant No. 2024JJ6299 \& No. 2026JJ30002), the Scientific Research Fund of Hunan Provincial Education Department (Project No. 25B0095), and NSFC Grant No. 12101226 \& No. 12371071.}


\begin{document}

\begin{abstract}
We investigate the geometric behavior of $\tau(E)$ for bounded finite-perimeter sets $E \subset \mathbb R^n$, where $\tau(E)$ is the trace constant introduced by Figalli--Maggi--Pratelli [Invent. Math. 2010]. This quantity is a key ingredient in proving a quantitative isoperimetric inequality with the optimal exponent.

We first show that for every $\ez>0$ one can find a bounded open set $\Omega \subset \mathbb R^n$ that is very close to the unit ball $\mathbb B^n$ in the sense that
$$
\tau(\mathbb B^n)>\tau(\Omega)>\tau(\mathbb B^n)-\ez \quad \text{and} \quad P(\Omega \Delta \mathbb B^n)\le C(n)\epsilon,
$$
while at the same time the complement of $\Omega$ has infinitely many connected components. Thus, $\tau(\Omega)$ can be made arbitrarily close to $\tau(\mathbb B^n)$ even when $\Omega$ has highly intricate geometry.

We then establish, under a mild additional hypothesis, the equivalence between a condition formulated in terms of $\tau$ and two classical criteria from the literature for open sets that admit trace inequalities. As a consequence, we obtain the John-type characterization of domains that support a trace inequality, assuming the ball separation property.
\end{abstract}

\maketitle
 
\section{Introduction}

In their seminal work \cite{FMP2010}, Figalli, Maggi and Pratelli applied a mass transportation approach to show a quantitative version of the isoperimetric inequality, a high-profile topic, stated as follows:
For any set of finite perimeter $E\subset\mathbb R^n$ with $|E|=|\mathbb B^n|$, where $\mathbb B^n$ denotes the standard unit Euclidean ball in $\mathbb R^n$, it holds that 
\begin{equation}\label{quan}
    P(E)-P(\mathbb B^n)\ge c(n)\min_{x\in\mathbb R^n}|E\Delta (x+\mathbb B^n)|^2.
\end{equation}
 Remarkably, the constant $c(n)>0$ is independent of the set $E$.
Here the Lebesgue measure $\mathcal L^n$ of the set $E\subset\mathbb{R}^n$ is denoted by $\vert E\vert$, and the symmetric difference between $E$ and $x+\mathbb B^n$ is denoted by $E\Delta(x+\mathbb B^n)$. For any function $u$ that is locally integrable on $E$ and whose distributional derivative is a Radon measure, we denote the total variation measure of $u$ by $\|Du\|(E)$, and define the perimeter of $E$  as 
$$P(E):=\|D\chi_E\|(\mathbb R^n);$$
see Section \ref{prelim} for preliminaries and more specific definitions. 

A crucial step in their proof is to show that, any set of finite perimeter $E\subset\mathbb R^n$ can be modified to a new (open) set $G$, which in particular satisfies
$
      \tau(G)\ge 1+\delta_0,
$
where the definition of $\tau$ is given as follows; we refer to \cite[Theorem 3.4]{FMP2010} for more detailed information. 
\begin{defn}\label{trace cons}
    Let $E\subset \mathbb R^n$ be a set of finite perimeter with $0<|E|<+\infty$. The trace constant $\tau(E)$ of $E$ is defined by setting
    $$\tau(E):=\inf\left\{\frac{P(F)}{\mathcal{H}^{n-1}(\partial^* E\cap \partial^* F)}:F\subset E, 0<|F|\le \frac{|E|}{2}\right\}\ge 1.$$
\end{defn}

The constant $\tau(\mathbb B^n)$ has already  been studied in \cite{CFNT2018}. To be more specific, let $O$ be the origin and $B(P,r)$ be the ball with radius $r>0$ and center $P$ located at the negative $x_1$-half-axis, such that 
$\partial B(P,r)\cap \mathbb S^{n-1}\neq \emptyset$. Choose a point $Q\in \partial B(P,r)\cap \mathbb S^{n-1}$.
Let 
$$E_{\varphi,\vartheta}=\mathbb B^n\setminus \overline{B}(P,\,r)$$
be the half-moon shaped set as in Figure \ref{fig:example1}, where $\vartheta\in (0,\pi)$ denotes the angle between the vector $\overrightarrow{PO}$ and the vector $\overrightarrow{OQ}$, and
 $\varphi\in (0, \vartheta)$ denotes the angle between 
$\overrightarrow{PO}$ and the vector $\overrightarrow{PQ}$.

Observe that the operator $K_{\rm med}$ defined in \cite[(2.25)]{CFNT2018} satisfies
$$\tau(E)=\frac{1}{K_{\rm med}(E)}+1. $$
Then \cite[Theorem 2.3]{CFNT2018} together with \cite[Theorem 4.1]{CFNT2018} implies the following. 

\begin{thm}[{\cite[Theorem 4.1]{CFNT2018}}]\label{cite cianchi}
    The infimum in the definition of $\tau(\mathbb B^n)$ is attained when $F$ is exactly taken as $E_{\varphi, \vartheta}$, where $ \varphi$ and $ \vartheta$ satisfy the condition \cite[(4.2)]{CFNT2018} with $\varphi_\rho=\varphi,\, \vartheta_\rho=\vartheta$ and $\rho=1/2$ there. In particular, one has 
\begin{equation}\label{tau b}
    \tau(\mathbb B^n)=\frac{P(E_{ \varphi, \vartheta})}{\mathcal{H}^{n-1}(\partial^* E_{ \varphi,\vartheta}\cap \mathbb{S}^{n-1})}>1, 
\end{equation}
and via \cite[Lemma 4.3]{CFNT2018}, 
\begin{equation}\label{volumn equi}
    |E_{ \varphi, \vartheta}|=|\mathbb B^n\setminus E_{ \varphi, \vartheta}|.
\end{equation}
\end{thm}

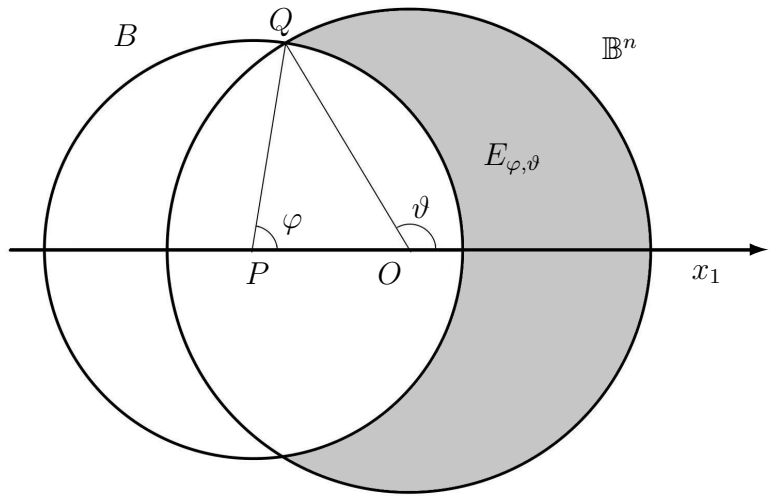
\begin{figure}[htbp] 
  \centering
\includegraphics[width=0.6\textwidth]{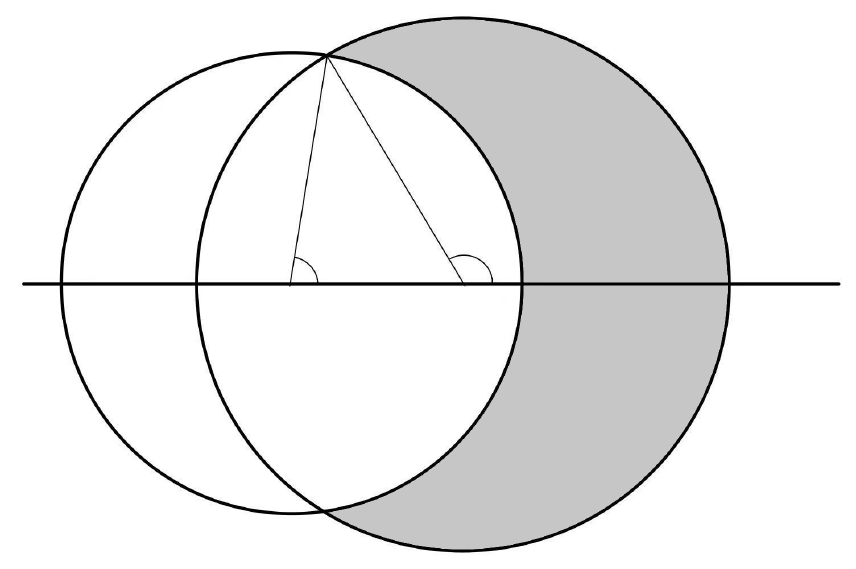}
   \caption{The half-moon shaped set $E_{\varphi,\vartheta}$ is colored in grey.}
   \label{fig:example1}
   \begin{picture}(0,0)(250,-10)
       \put(150,205){$B$} \put(250,115){$O$} \put(200,115){$P$} \put(335,200){$\mathbb B^n$} \put(263,140){$\vartheta$} \put(214,136){$\varphi$} \put(290,160){$E_{\varphi,\vartheta}$}
       \put(209, 211){$Q$}
       \put(369,117){$x_1$}
         \linethickness{1.0pt}
  \put(111,128){\vector(1,0){287}} 
   \end{picture}
\end{figure}

Let $\Omega\subset \mathbb R^n$ be an open set.  We say that $\Omega$ supports a $(1,1)$-trace inequality for functions of bounded variation if there exists a bounded linear operator $T: BV(\Omega)\to L^1(\partial\Omega)$, such that 
\begin{equation}\label{11trace}
\inf_{c\in \mathbb R}\int_{\partial\Omega}|Tu-c|\,d\mathcal{H}^{n-1}\le C_T\|Du\|(\Omega)\quad 
\text{for any }u\in BV(\Omega)
\end{equation}
holds for some positive constant $C_T$ independent of $u$. Then thanks to  $\tau(G)\ge 1+\delta_0$, one  gets that $G$ supports a trace inequality of the following form:
$$\inf_{c\in\mathbb R^n}\int_{\partial^* G}|tr_{G}u-c|\,d\mathcal{H}^{n-1}\le C(n)  {\delta_0}^{-1}\|Du\|(G^{(1)})\quad \text{for any }u\in BV(\mathbb R^n)\cap L^{\infty}(\mathbb R^n),$$
 where $tr_{G}u$, called the inner trace of $u$ on $G$, is defined on the reduced boundary $\partial^* G$; see \cite[Lemma 3.1]{FMP2010}. 
The  construction of $G$ is rather subtle and is obtained via finding a maximal critical set in $E$ (see \cite[Lemma 3.2]{FMP2010}). According to the trace inequality above, with some nontrivial effort  one can finally deduce that the original set $E$ satisfies \eqref{quan}.

For classical treatments of traces of $BV$-functions in Euclidean spaces, we refer to \cite[Chapter 3]{AFP2000} and \cite[Chapter 5]{EG92}. Recent developments in metric measure spaces can be found in works such as \cite{BM2021, L2023, LS2018, MSS2018}.
The aim of the current manuscript is to further study the geometric properties of the trace constant $\tau$.

We start with the following example, telling that without any further assumptions,
the geometry of a set could be bad even if the trace constant of such a set is close to $\tau(\mathbb B^n)$. This example was briefly mentioned in \cite[Page 578]{BK1995} for Sobolev--Poincar\'e inequalities in  planar domains. 

\begin{thm}\label{counexample}
 For any $\epsilon>0$, there exists an open set $\Omega\subset \mathbb R^n,\, n\ge 2$ of finite perimeter, $|\Omega|<+\infty$,  $\mathcal H^{n-1} (\partial\Omega\setminus\partial^*\Omega)=0$, close to the unit ball in the sense that
 $$\Omega= \mathbb B^n\setminus D,\quad P(\Omega\Delta \mathbb B^n)\le C(n)\epsilon,$$
 together with 
  $$\tau(\mathbb B^n )>\tau(\Omega)>\tau(\mathbb B^n)-\epsilon.$$   
Moreover, up to a modification on a set of $\mathcal L^{n}$-measure zero,
$D$  consists of countably infinitely many balls. 
\end{thm}

Regarding the geometric characterization of domains supporting the $(1,1)$-trace inequality \eqref{11trace}, a lot of pioneering work has already been done. For instance, under the additional assumptions
\begin{equation}\label{almost theoretic}
            P(\Omega)<+\infty,\quad 0<|\Omega|<+\infty\quad\text{and}\quad\mathcal{H}^{n-1}(\partial \Omega\setminus\partial^{*}\Omega)=0,
        \end{equation}
 Maz'ya established a sufficient and necessary condition, in a   similar manner to  
 $$\tau(\Omega)\ge 1+\delta_0,$$
 for the open set $\Omega$  supporting $(1,1)$-trace inequality \eqref{11trace} \cite[Theorem 9.6.4]{M2011}; see also Lemma \ref{mazya s and n} below. 

Meanwhile, analogous to \cite[Theorem 9.6.4]{M2011}, Ziemer introduced another sufficient condition for those bounded domains supporting \eqref{11trace}; see \cite[Theorem 5.10.7]{Z1989}. Such domains are called {\it admissible domains}; see Definition \ref{admissible} for detailed information.


For any set $\Omega$ satisfying the conditions \eqref{almost theoretic}, we show that  $\Omega$ can be assumed to be open up to a modification on a set of $\mathcal L^n$-measure zero; see Lemma \ref{open rep} for a detailed explanation. This allows us to study the equivalence between the following three conditions under the assumption \eqref{almost theoretic}. 

\begin{thm}\label{equiva}
    Let $\Omega\subset\mathbb R^n, \, n\ge 2$ be a set of finite perimeter satisfying \eqref{almost theoretic}. Then up to a modification on a set of $\mathcal L^n$-measure zero of $\Omega$, the following statements are equivalent:
    \begin{enumerate}[(i)]
  \item
    The trace constant $\tau(\Omega)\ge 1+\delta_0$ for some constant $\delta_0>0$;

    \item
    $\Omega$ is an open set satisfying \eqref{11trace} with constant $C_T>0$;

     \item 
      $\Omega$ is a $\Theta$-admissible domain.
    \end{enumerate}
  The dependence of the constants is given as follows:  If $\Omega$ satisfies (i), then $\Omega$ meets (ii) with constant $C_T=(\tau(\Omega)-1)^{-1}$; if $\Omega$ satisfies (ii), then $\Omega$ meets (iii) with constant $\Theta=C_T$; if $\Omega$ satisfies (iii), then $\Omega$ meets (i), with constant $\tau(\Omega)$ (or more precisely, the lower bound of $\delta_0$) depending on both $n$ and $\Omega$ itself.  
\end{thm}

\begin{rem}
    In (iii)$\Rightarrow$(i) of Theorem \ref{equiva}, the dependence  of the constant $\tau(\Omega)$ cannot be determined solely by $\Theta$. An example is given as follows. Consider the rectangle
    $$R_l:=(-l,l)\times (-1,1)\subset \mathbb R^2,$$
    which is a $\sqrt{2}$-admissible domain for any $l\gg 1$. However, as $l\to\infty$, $\tau(R_l)\to 1^+$.  
\end{rem}

Let us recall the definition of a John domain.
  \begin{defn}\label{John}
    For   $J\ge 1$, a (bounded) domain $\Omega\subset \mathbb R^n$ is said to be  $J$-John if there exists a distinguished point $x_0\in \Omega$ so that, for any $x\in \Omega$, there exists a curve $\gamma\subset \Omega$ joining $x$ to $x_0$ satisfying the following condition: 
\begin{equation}\label{John curve}
\ell(\gamma[x,\,y])\le J\dist(y,\,\partial \Omega) \quad \text{ for any $y\in \gamma$}, 
\end{equation}
where $\ell(\gamma[x,\,y])$ denotes the length of the subcurve of $\gamma$ joining $x$ to $y$ with respect to the metric induced by the standard Euclidean norm $|\cdot|$.
We usually call $x_0$ the  {John center} of $\Omega$ and $\gamma$ the  {John curve} joining $x_0$ and $x$. 
\end{defn}

Partially motivated by \cite{FMP2010}, the authors presented an alternative proof of \eqref{quan} (for the general anisotropic perimeter)  in a recent manuscript \cite{SZ2024}, by utilizing the John property of (almost) minimal surfaces \cite{DS1998}. Specifically speaking, the authors modified each set $E$ sufficiently close to $\mathbb B^n$ via a variational method, known as selection principle, to obtain a new set $\Omega$ with uniformly bounded John constant $J$ independent of $E$. The set $\Omega$ also satisfies  a uniform volume density estimate and supports the $(1,1)$-trace inequality \eqref{11trace} with the constant only depending on the dimension $n$ and $J$, thus supporting \eqref{quan}; see \cite[Theorem 1.3]{SZ2024} for detailed discussion. 
 
Partially motivated by this, we delve into examining the relationship between an open set that supports the $(1,1)$-trace inequality \eqref{11trace} and its John property.
 Observe that,  Theorem \ref{counexample} already yields that sets satisfying both \eqref{almost theoretic} and supporting the $(1,1)$-trace inequality \eqref{11trace} could have overly complicated geometric properties. 
Nevertheless, when $\Omega$ satisfies some extra mild geometric assumption, such as ball separation property, then one can show that $\Omega$ is a John domain, together with an upper density bound on the perimeter.

Let us firstly recall the definition of the ball separation property.

\begin{defn}\label{ball sep}
    We say that a domain $\Omega\subset \mathbb R^n$ has a ball separation property with respect to a distinguished point $x_0\in \Omega$ if there exists a constant $S>0$ such that  
    for each $z\in \Omega$, one can find a curve $\gamma\subset \Omega$ joining $z$ to $x_0$ so that, for any point $\eta\in \gamma$, either $$\gamma[z,\eta]\subset B(\eta, S\dist(\eta,\partial \Omega))=:B,$$
    or each point $y\in \gamma[z,\eta]\setminus B $ belongs to a different component of $\Omega\setminus \partial B$ than $x_0$. 
\end{defn}

Note that, each finitely connected planar domain has a ball separation property; see  e.g. \cite{BK1995} (and further in \cite[Page 75]{BHK2001} and \cite[Page 191]{KRZ2017}). Now we can state the following result, as a byproduct of the proof of Theorem \ref{equiva}.

\begin{cor}\label{iff}
    Suppose that the open set $\Omega\subset\mathbb R^n$ satisfies 
    \eqref{almost theoretic}.  
    \begin{enumerate}[(i)]
        \item Assume that $\Omega$ satisfies the ball separation property with respect to a distinguished point $x_0\in \Omega$ and the constant $S>0$.  Moreover, for any $u\in BV(\Omega)$, the trace $Tu$ exists on $\partial \Omega$ and \eqref{11trace} holds. Then by letting 
        $$r_0:=\dist (x_0,\Omega^c),$$
        $\Omega$ is a $J$-John domain with John center $x_0$, where $C_T$ is the same constant in \eqref{11trace} and $J=J(n, C_T, S, |\Omega|/r_0^n)$. In addition, $\partial\Omega$ satisfies the upper density estimate
\begin{equation}\label{hau cond 1}
    \mathcal H^{n-1}(\partial \Omega\cap B(x,r))\le C_H r^{n-1}\quad \text{ for any }r\in (0,r_0)\text{ and }x\in \partial\Omega,
\end{equation}
where $C_H=C_H(n, C_T, r_0, P(\Omega))$ is a positive constant.

        \item Conversely, if $\Omega$ is a $J$-John domain with John center $x_0\in\Omega$ and satisfies \eqref{hau cond 1}, then $\Omega$ has a ball separation property with a distinguished point $x_0$ and constant $S>0$, where $S=J$. Furthermore, for each $u\in BV(\Omega)$, the trace $Tu(x)$ exists for $\mathcal H^{n-1}$-a.e. $x\in \partial\Omega$, and there exists a constant $C_T=C_T(n, J, C_H, r_0)>0$ such that \eqref{11trace} holds. 
    \end{enumerate}
\end{cor}

\begin{rem}
    The ball separation property is necessary for the equivalence of John property and trace inequality. In fact, the example, constructed in Chapter 3 to prove Theorem \ref{counexample}, is precisely a domain that supports the $(1,1)$-trace inequality yet fails to have the ball separation property; see Lemma \ref{ch3 fin per} and Proposition~\ref{tau value} together with Theorem~\ref{equiva}.
\end{rem}

Our paper is structured as follows: In Section \ref{prelim}, we recall some elementary  results, especially the definition and  properties of admissible domains. 
Subsequently, in Section 3, we construct the set $\Omega$ stated in Theorem \ref{counexample} and then verify \eqref{almost theoretic} for $\Omega$  in Lemma \ref{ch3 fin per}. Furthermore, we prove in Proposition \ref{tau value} that $\tau(\Omega)$ is sufficiently close to $\tau(\mathbb B^n)$, thereby concluding Theorem \ref{counexample}. Finally, in Section 4,  we show that each open set satisfying \eqref{almost theoretic} and the $(1,1)$-trace inequality \eqref{11trace} necessarily have connectivity and boundedness, thus obtaining the equivalence of three types of sets referred to in Theorem \ref{equiva}. In addition, Corollary \ref{iff} follows as a consequence of (the proof of) Theorem \ref{equiva}.

\section{Preliminary}\label{prelim}

\subsection{Notations and some elementary lemmas.}

To begin with, we establish the notation as follows:
For a (rectifiable) curve $\gamma\subset \mathbb R^n$, we denote the length of $\gamma$ by $\ell(\gamma)$. For any pair of points $x,y \in \gamma$, we denote a subcurve of $\gamma$ joining $x$ to $y$ by $\gamma[x,y]$.

For any set $E\subset\mathbb{R}^n$, the closure of $E$ with respect to the Euclidean topology is denoted as $\overline{E}$ or $cl(E)$ and its complement is denoted as $E^c:=\mathbb R^n\setminus E$ for brevity. The interior of $E$ is denoted as $\inter(E)$. The Lebesgue measure of $E$ is denoted by $|E|$ and the $s$-dimensional Hausdorff measure of $E$ is denoted by $\mathcal H^s(E)$. A general constant is denoted by $C$, which may vary between different estimates, and we include all the constants it depends on within the parentheses, denoted as $C(\cdot)$.

In addition, the standard Euclidean norm is denoted by $|\cdot| $. With this notation, for any set $E\subset\mathbb R^n$, we write 
$$\dist(x,E):=\inf_{y\in E}|x-y|.$$
We use 
$$B(x,r):=\{z\in\mathbb R^n:|z-x|<r\}$$
to represent the ball with center $x$ and radius $r$. The closure of $B(x,r)$ is denoted by $\overline{B}(x,r)$. For brevity, we write $B_r=B(0,r)$. In particular, we denote by $\mathbb B^n$ the unit ball $B_1\subset\mathbb R^n$. Its boundary and its Lebesgue measure are denoted by $\mathbb S^{n-1}$ and $\omega_n$, respectively.

For any $\mathcal L^n$-measurable set $E\subset\mathbb R^n$ and any function $u$ integrable on $E$, the integral average  of $u$ over $E$ is denoted by 
$$u_E:=\bint_E u(x)\,dx=\frac{1}{|E|}\int_Eu(x)\,dx.$$

A function $f\in L^1(U)$ has bounded variation on the open set $U\subset \mathbb R^n$ if its total variation is finite, i.e., $\|Df\|(U)<+\infty$. We denote the space of functions of bounded variation in $U$ by $BV(U)$. This space is equipped with the norm
$$\|f\|_{BV(U)}:=\|f\|_{L^1(U)}+\|Df\|(U).$$

For any measurable set $E\subseteq \R^n$, the distributional gradient $D\chi_E$ of its characteristic function $\chi_E$ induces the relative
perimeter of $E$ with respect to the measurable set  $V\subset\mathbb R^n$, which is defined as 
$$P(E;V)=\|D\chi_E\|(V).$$
In particular, when $V=\mathbb R^n$, $P(E;\mathbb R^n)$ is the so-called the \emph{perimeter} of $E$, denoted by $P(E)$ for simplicity.

A measurable set $E\subset \mathbb R^n$
is said to be of \emph{finite perimeter} provided
that  $D\mathbf{\chi}_E$ is a Radon measure with finite total variation, i.e., $\|D\mathbf{\chi}_E\|(\R^n)<\infty$.
According to the Lebesgue-Besicovitch differentiation theorem for measures (cf. \cite[Theorem 1.34]{EG92}), for $\|D\mathbf{\chi}_E\|$-a.e. $x$, it holds that
$$
\lim_{r\to 0^+} - \frac{D\mathbf{\chi}_E(x+rB^n)}{\|D\mathbf{\chi}_E\|(x+rB^n)} = \nu_E(x) \quad \text{and} \quad |\nu_E(x)| = 1.
$$
The set of points $x$ where this condition holds is called the \emph{reduced boundary} of $E$ and is denoted by $\partial^* E $.
 At points on the reduced boundary, $\nu_E(x)$ represents the \emph{measure-theoretic outward unit normal} to $E$ at $x$.
Also, {up to modifying $E$ in a set of Lebesgue measure zero}, one can assume that $\overline{\partial^* E}=\partial E$ \cite[Proposition 12.19 \& (15.3)]{M2012}.
Additionally, $\|D\mathbf{\chi}_E\|$ coincides with $\mathcal{H}^{n-1}\llcorner \partial^*E$, the Hausdorff measure $\mathcal{H}^{n-1}$ restricted on $\partial^*E$. 

For any sets $E$ of finite perimeter with $|E|<\infty$, the isoperimetric inequality 
\begin{equation}\label{iso 2}
 P(E)\ge n\omega_n^{1/n}|E|^{(n-1)/n} 
\end{equation}
 holds. Analogously, the relative isoperimetric inequality states that  for each ball $B(x,r)\subset \mathbb R^n$,
\begin{equation}\label{rela iso}
    P(E;B(x,r))\ge C_{re}\min\{|B(x,r)\cap E|,|B(x,r)\setminus E|\}^{1-\frac{1}{n}},
\end{equation}
 where $C_{re}=C_{re}(n)>0$; see \cite[Theorem 5.11]{EG92} for these two inequalities. For more details on sets of finite perimeter, we refer the interested reader to  \cite[Sections 12 and 15]{M2012}. Besides, for  any Borel set $E\subset \mathbb R^n$, the Fleming-Rishel Coarea Formula tells that 
\begin{equation}\label{fr coarea}
    \|Df\|(E)=
    \int_{\mathbb R}P(\{f>t\};E)\,dt\quad \text{for any }f\in BV(\mathbb R^n);
\end{equation}
see \cite[(2.22)]{FMP2010}.

 For a Borel set $E\subset \mathbb R^n$ and 
$\lambda\in [0,1]$, we denote by $E^{(\lambda)}$ the set 
\begin{equation}\label{density set}
    E^{(\lambda)}:=\left\{x\in\mathbb R^n: \lim_{r\to 0}\frac{|E\cap B(x,r)|}{|B(x,r)|}=\lambda\right\}.
\end{equation}
Then the \emph{measure-theoretic boundary} $\partial_M E$ of $E$ is the set defined as $\partial_M E:=\mathbb R^n\setminus (E^{(0)}\cup E^{(1)})$. 
Federer's theorem characterizes the relation among $\partial^* E,\,\partial_M E$ and $E^{(1/2)}$. 
\begin{lem}[Federer's theorem]\label{Federer}
    If $E\subset\mathbb R^n$ is a set of locally finite perimeter, then $$\partial^*E\subset E^{(1/2)}\subset\partial_M E\quad \text{and}\quad 
    \mathcal{H}^{n-1}(\partial_M E\setminus\partial^*E)=0.$$
\end{lem}





In addition, we record the following lemma regarding set operations; see \cite[Theorem 16.3]{M2012} and \cite[Exercise 16.6]{M2012}.
\begin{lem}\label{Maggi}
    Suppose that $E$ and $F$ are sets of finite perimeter in $\mathbb R^n$. For brevity, we let 
    $$\{\nu_E=\nu_F\}:=\{ x\in \partial^* E\cap \partial ^*F: \nu_E=\nu_F\}\quad \text{and}\quad\{\nu_E=-\nu_F\}:=\{ x\in \partial^* E\cap \partial ^*F: \nu_E=-\nu_F\}.$$
    Then $E\cap F$, $E\setminus F$ and $E\cup F$ are sets of finite perimeter, satisfying that for every Borel set $G\subset\mathbb R^n$, 
    $$P(E\cap F; G)=P(E; F^{(1)}\cap G)+P(F; E^{(1)}\cap G)+\mathcal{H}^{n-1}(\{\nu_E=\nu_F\}\cap G),$$
    $$P(E\setminus F; G)=P(E; F^{(0)}\cap G)+P(F; E^{(1)}\cap G)+\mathcal{H}^{n-1}(\{\nu_E=-\nu_F\}\cap G),$$
    $$P(E\cup F; G)=P(E; F^{(0)}\cap G)+P(F; E^{(0)}\cap G)+\mathcal{H}^{n-1}(\{\nu_E=\nu_F\}\cap G).$$
Moreover, 
$$\partial^* (E\cap F)= (\partial^*E \cap F^{(1)})\cup (\partial^*F \cap E^{(1)})\cup \{\nu_E=\nu_F\}. $$
In particular, if $E\subset F$, then $\mathcal{H}^{n-1}(\partial^*F\cap \partial^*E)=\mathcal{H}^{n-1}(\{\nu_E=\nu_F\}).$
\end{lem}

\subsection{Admissible domains.}

In \cite[Theorem 9.6.4]{M2011}, Maz'ya introduces a sufficient and necessary condition for an open set to support the trace inequality. {We state it as the following lemma. }
\begin{lem}\label{mazya s and n}
    Let $\Omega\subset\mathbb R^n$ be an open set satisfying \eqref{almost theoretic}.  
    \begin{enumerate}
  \item[(i)]  
If for any $\mathcal{L}^n$-measurable set $F\subset \Omega$, 
\begin{equation}\label{admis 2 rep}
    \min\{P(F;\Omega^c), P(\Omega\setminus F;\Omega^c)\}\le C_M P(F;\Omega),
\end{equation}
where $C_M$ is a positive constant independent of $F$, then for any $u\in BV(\Omega)$, the trace $Tu$ exists. Moreover, \eqref{11trace} holds by replacing $C_T$ in \eqref{11trace} with  $C_M$.

 \item[(ii)]
If, for any $u\in BV(\Omega)$, the trace $Tu$ exists $\mathcal H^{n-1}$-a.e. on $\partial\Omega$ and \eqref{11trace} holds for some constant $C_T>0$, then for any $\mathcal{L}^{n}$-measurable set $F\subset\Omega$, \eqref{admis 2 rep} holds with $C_M=C_T$.
 \end{enumerate}
\end{lem}
By utilizing \eqref{admis 2 rep}, one can show that $\Omega$ is bounded and connected; see Lemma \ref{admis bounded} for a detailed discussion.
 
On the other hand, regarding the condition \eqref{admis 2 rep}, Ziemer introduces an alternative definition, which is named as \emph{admissible domain}; see \cite[Definition 5.10.1]{Z1989}. 
    
\begin{defn}\label{admissible}
    A bounded domain $\Omega\subset\mathbb R^n$ with finite perimeter is said to be a $\Theta$-admissible domain provided
      $\mathcal{H}^{n-1}(\partial \Omega\setminus\partial^{*}\Omega)=0$
        and that there exists a positive constant $\Theta>0$ such that, for each $x\in\partial\Omega$, there exists $r=r(x)>0$ such that 
         \begin{equation}\label{admissible 2}
         \mathcal H^{n-1}(\partial^{*} \Omega\cap \partial^* E)\le \Theta \mathcal H^{n-1}(\partial^* E\cap \Omega),
         \end{equation}
         whenever $E\subset\overline{\Omega}\cap B(x,r)$  is a set of finite perimeter.
        
\end{defn}
\begin{rem}\label{theo reduce rep}
    It is noteworthy that in \cite[Definition 5.10.1]{Z1989}, \eqref{admissible 2} is formulated in terms of measure-theoretic boundary. Specifically, the reduced boundary $\partial^*\Omega$ and $\partial^*E$ in \eqref{admissible 2} are replaced by $\partial_M\Omega$ and $\partial_M E$, respectively, as in \cite[(5.10.1)]{Z1989}. Moreover, in \cite[(5.10.1)]{Z1989}, \eqref{admissible 2} is also required for all measurable subsets $E\subset\overline{\Omega}\cap B(x,r)$. Nevertheless, according to Federer's Theorem (see Lemma~\ref{Federer}), these formulations are equivalent whenever $E$ is a set of finite perimeter; otherwise, the term $ \mathcal H^{n-1}(\partial_M E\cap \Omega)=+\infty$ and thus \cite[(5.10.1)]{Z1989} holds for any bounded domain $\Omega$ with finite perimeter. 
\end{rem}

   We prove in Theorem \ref{equiva} that the open set with \eqref{admis 2 rep} is actually an admissible domain. Recall that admissible domain is also a BV-extension domain; see \cite[Corollary 1, Page 497]{M2011}.

\begin{lem}\label{Mazya-extension}
Let $\Omega\subset\mathbb R^n$ be a set of finite perimeter which is open and satisfies 
\eqref{almost theoretic}.
 If, for any measurable set $F\subset \Omega$, the inequality 
\eqref{admis 2 rep}
 holds, then for any $u\in BV(\Omega)$, there exists a constant $c>0$ and an extension $\hat u\in BV(\mathbb R^n)$ satisfying that  $\hat u(x)=u(x)$ whenever $x\in \Omega$ and $\hat u(x)=c$ whenever $x\in \Omega^c$, such that 
 \begin{equation}\label{ext}
     \|D\hat u\|(\mathbb R^n)\le (C_M+1)\|Du\|(\Omega)
 \end{equation}
follows, where $C_M$ is the same constant in \eqref{admis 2 rep}.
\end{lem}

We finally record the following result by Buckley and Koskela.

\begin{lem}[\cite{BK1995}]\label{Sobolev John}
     Suppose that $\Omega\subset \mathbb R^n$ is a domain with $|\Omega|<+\infty$ that supports a $(p^*,p)$-Poincar\'e inequality for $1\le p<n$, i.e. 
\begin{equation}\label{Poincare}
  \inf_{c\in\mathbb R}\left(\int_{\Omega}|u-c|^{p^*}\,dx\right)^{1/p^*}\le C_p\left(\int_\Omega|Du|^p\,dx\right)^{1/p}\quad \text{ for any }u\in W^{1,p}(\Omega), 
    \end{equation}       
where $p^*=\frac{np}{n-p}$, and that $\Omega$ satisfies the ball separation property with distinguished point $x_0\in \Omega$ and the constant $S>0$ as in Definition \ref{ball sep}. Then $\Omega$ is a $J$-John domain with John center $x_0$, where $$J=J(n, p, C_p, S, |\Omega|/\dist(x_0,\Omega^c)^n).$$
\end{lem}

\section{Proof of Theorem \ref{counexample}}

In this section, we construct the domain $\Omega\subset\mathbb R^n$ in accord with the statement of Theorem \ref{counexample}. This domain $\Omega$ supports the trace inequality \eqref{11trace}, but it fails to be John and does not have the ball separation property.

For brevity, we introduce the following terminology: A continuous function $\omega:\mathbb R^+\to \mathbb R^+$ is a \emph{modulus of continuity} if $\omega$ is non-decreasing and satisfies $
\lim_{t\to 0^+}\omega(t)=0$.

For $k\in \mathbb N^+$, let $E_k$ be the non-empty set consisting of $C(n)(k!)^{n-1}$ distinct points on the sphere $\partial B_{1-2^{-k}}$ with mutual distance at least $n^{-1}({k!})^{-1}$ such that, $E_k$ is symmetric with respect to the hyperplane $\{x_1=0\}$ and satisfies
 \begin{equation}\label{dense bo sep}
     \dist(z, E_k)\le 2\sqrt{n}(k!)^{-1}\quad \text{for each }z\in \partial B_{1-2^{-k}}.
 \end{equation}
 
Given $x\in E_k$, for $0<\delta\le n^{-2}$, we write
$$B_{x,k}:= B\left(x,\frac{2^{-k-2}} {(k!)^n}\right),\quad 2B_{x,k}:=B\left(x,\frac{2^{-k-1}} {(k!)^n}\right)\quad\text{ and } \quad B_{x,k}^\delta:=B\left(x,\frac{2^{-k}}{(k!)^{1/\delta}}\right)$$
for brevity. In the following lemma, for any $c\in\mathbb R$, we write
$$\{x_1\le c\}=\{x=(x_1,\cdots,x_n)\in\mathbb R^n: x_1\le c\}.$$
 We define $\{x_1=c\}$ and $\{x_1\ge c\}$ in similar ways.
\begin{lem}\label{k0}
    Let $E_{\varphi,\vartheta}= \mathbb B^n\setminus \overline{B}(P,r)$ be the  sets defined  in Theorem \ref{cite cianchi}, where $B(P,r)$ is a ball with radius $r$ and center $P$ located at the negative $x_1$-half-axis. Then
    there exists an integer $k_0\ge 2$, depending only on the dimension $n$, such that for any integer $k\ge k_0$ and $\delta\in (0, n^{-2})$, 
 \begin{equation}\label{volume est layer}
     \sum_{x\in E_k}|\overline{B}_{x,k}^\delta\cap E_{\varphi,\vartheta}| \ge \frac{1}{2}\sum_{x\in E_k}|\overline{B}_{x,k}^\delta|,
 \end{equation}
and there exists a point $z\in E_k$ satisfying 
 \begin{equation}\label{bo in halfmoon0}
     \overline{B}_{z,k}^\delta \subset \subset E_{\varphi,\vartheta}.
 \end{equation}
\end{lem}
\begin{proof}
We first claim that, there exists $d_n=d_n(n)>0$ small for which
    \begin{equation}\label{2 fact}
        E_{\varphi,\vartheta} \cap \{x_1\le-d_n\}\neq \emptyset\quad \text{and}\quad O\in  B(P,r). 
    \end{equation}
     Indeed, we first suppose the first relation in \eqref{2 fact} fails. Then one has $E_{\varphi,\vartheta}\subset \{x_1\ge 0\}$. This immediately yields 
 $|E_{\varphi,\vartheta}|<|\mathbb B^n|/2$, which contradicts to \eqref{volumn equi}.
 
 We next suppose that, on the contrary, $O\notin \inter(B(P,r))$. Then since $P$ is located at the negative $x_1$-half-axis, $\overline{B}(P,r)\subset \{x_1\le 0\}$. This immediately yields 
$$|\overline{B}(P,r)\cap \mathbb B^n|=|\mathbb B^n \setminus E_{\varphi,\vartheta}|<|\mathbb B^n|/2,$$
 which contradicts to \eqref{volumn equi}. Hence, we establish \eqref{2 fact}.

Since $E_k$ is non-empty and symmetric with respect to $\{x_1=0\}$, for any $k\in \mathbb N^+$ large enough,
\begin{equation}\label{over half}
    \mathcal H^{0}(E_{k}\cap \{x_1>-d_n/3\})\ge \frac{1}{2}\mathcal H^{0}(E_{k})\ge 1.
\end{equation}
In particular, $E_{k}\cap \{x_1>-d_n/3\}\neq \emptyset$.

We claim that, there exists $k_0\in \mathbb N^+$ depending only on $n$, such that for any $k\ge k_0$ and $0<\delta<n^{-2}$, every $x\in E_{k}\cap \{x_1>-d_n/3\}$ satisfies
\begin{equation}\label{bo halfmoon}
    \overline{B}_{x,k}^\delta \subset \subset E_{\varphi,\vartheta}.
\end{equation}
Assuming that this claim holds for the moment, then we obtain \eqref{bo in halfmoon0} and, via \eqref{over half},  
\begin{align}\label{volume half2}
     \sum_{x\in E_k}|\overline{B}_{x,k}^\delta\cap E_{\varphi,\vartheta}| & \ge  \sum_{\overline{B _{x,k}^\delta}\subset E_{\varphi,\vartheta}}|\overline{B}_{x,k^\delta}| \ge  \sum_{x\in E_k\cap \{x_1>-d_n/3\}}|\overline{B}_{x,k}^\delta|\ge  \frac{1}{2}\sum_{x\in E_k}|\overline{B}_{x,k}^\delta|, 
\end{align}
which  implies \eqref{volume est layer}.

Now we proceed to show \eqref{bo halfmoon}.  Due to the convexity of $\overline{B}(P,r)$,  \eqref{2 fact} tells 
$$\{x_1=-d_n/2\}\cap \mathbb B^n\cap \overline{B}(P,r)\neq \emptyset.$$
Hence, since $\{x_1=-d_n/2\}\cap \mathbb S^{n-1}\subset (\overline{B}(P,r))^c$ follows from the first equality in \eqref{2 fact}, there exists a point $Z\in \{x_1=-d_n/2\}\cap \partial B(P,r) \cap \mathbb B^n$  and  $k_1=k_1(n)\in \mathbb N^+$ such that, for any point $M\in \{x_1> -d_n/2\}\cap(\mathbb B^n\setminus B_{1-2^{-k_1}})$ with $M=(M_1,\cdots,M_n)$,
\begin{equation}\label{k0 req}
  |\overrightarrow{OM}| \ge  1-2^{-k_1}> |\overrightarrow{OZ}|.
\end{equation}

Recall $P$ is on the negative part of $x_1$-axis. Let $\eta_M$ be the angle between $\overrightarrow{OP}$ and $\overrightarrow{OM}$, and let $\eta_Z$ be the angle between $\overrightarrow{OP}$ and $\overrightarrow{OZ}$. 
Observe that
$$|\overrightarrow{OM}|\cos \eta_M=-M_1 \ \text{ and } \ |\overrightarrow{OZ}|\cos \eta_Z=d_n/2.$$ Thus, by \eqref{k0 req}, $M_1>-d_n/2$ and the law of cosines, we have
\begin{align}  
|\overrightarrow{PM}|^2 &=|\overrightarrow{OM}|^2+|\overrightarrow{OP}|^2-2|\overrightarrow{OM}||\overrightarrow{OP}|\cos \eta_M=|\overrightarrow{OM}|^2+|\overrightarrow{OP}|^2+2|\overrightarrow{OP}|M_1\nonumber\\
& > |\overrightarrow{OZ}|^2+|\overrightarrow{OP}|^2-|\overrightarrow{OP}|d_n=|\overrightarrow{OZ}|^2 +|\overrightarrow{OP}|^2- 2|\overrightarrow{OZ}||\overrightarrow{OP}|\cos \eta_Z= |\overrightarrow{PZ}|^2=r^2.
\end{align}
Hence, $M\in \mathbb B^n\setminus \overline{B}(P,r)=E_{\varphi,\vartheta}$. By the arbitrariness of $M$, we conclude that
\begin{equation}\label{over half2}
    \{x_1>-d_n/2\}\cap (\mathbb B^n\setminus B_{1-2^{-k_1}}) \subset E_{\varphi,\vartheta}.
\end{equation}

Further observe that there exists another integer  $k_0\ge k_1+1$ depending only on $n$ such that,  for any $k\ge k_0$ and any $\delta\in (0, n^{-2})$, every $x\in E_k\cap \{x_1> -d_n/3\}$ satisfies 
\begin{equation}\label{bo in ring}
    B_{x,k}^\delta\subset\subset \{x_1> -d_n/2\}\cap(\mathbb B^n\setminus B_{1-2^{-k_1}}).
\end{equation}
Then as a consequence of \eqref{over half2}, we obtain \eqref{bo halfmoon}, thus concluding the proof of Lemma \ref{k0}. 
\end{proof}
 Let $k_0$ be as in Lemma \ref{k0}, and define 
$$D^\delta=\bigcup_{k=k_0}^{\infty}\bigcup_{x\in E_k} \overline B_{x,k}^\delta\quad \text{and} \quad \Omega^\delta=\mathbb B^n\setminus D^\delta.$$
 Then $\Omega^\delta$ is constructed by deleting from the unit ball $\mathbb B^n$ the countably many balls $$\{B_{x,k}^\delta: x\in E_k,\, k\geq k_0\};$$ see Figure \ref{fig:example2}.

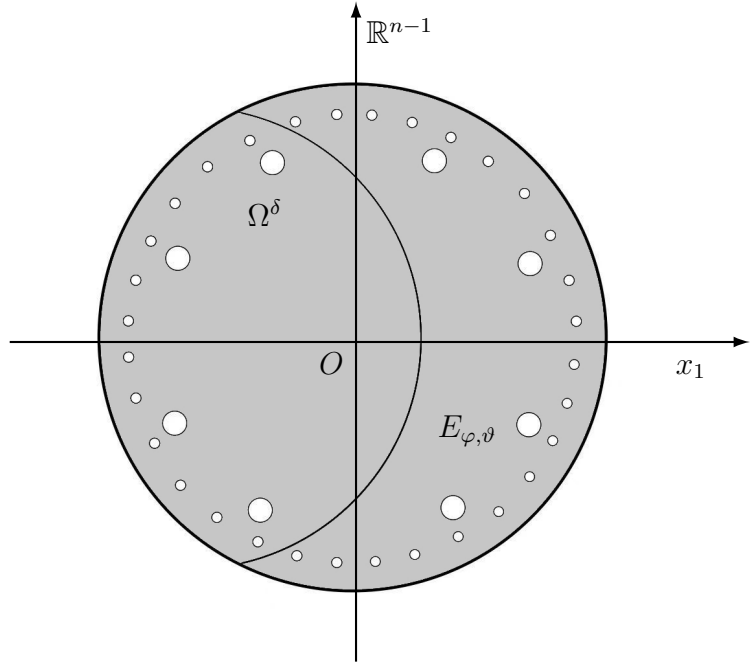
\begin{figure}[htbp]  
    \centering
    \includegraphics[width=0.5\textwidth]{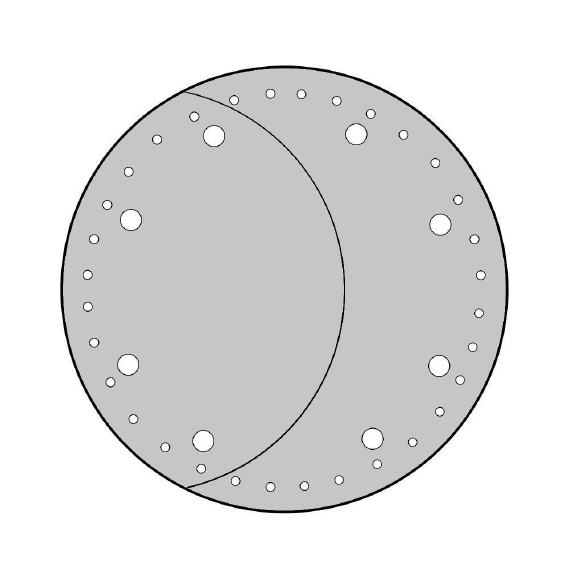}
    \caption{The set $\Omega^\delta$.}
    \label{fig:example2}
   \begin{picture}(0,0)(250,-10)
 \put(282,115){$E_{\varphi,\vartheta}$}
  \put(237, 139){$O$}
  \put(210, 195){$\Omega^\delta$}  
  \put(372,139){$x_1$}
\put(255,264){$\mathbb R^{n-1}$}
  \linethickness{0.8pt}
  \put(120,151){\vector(1,0){280}}   
  \put(251,30){\vector(0, 1){250}}  
   \end{picture}
\end{figure}

\begin{lem}\label{ch3 fin per}
The domain $\Omega^\delta$ has the following properties:
\item[(1)] $\Omega^\delta$ fails to be John and does not have the ball separation property. 
   \item[(2)] $\Omega^\delta$ is a set of finite perimeter with $\mathcal H^{n-1}(\partial\Omega^\delta\setminus\partial^*\Omega^\delta)=0$ and $0<|\Omega^\delta|<+\infty$. Moreover, 
   \begin{equation}\label{inn bound lem}
       \mathcal{H}^{n-1}(\partial^* D^\delta)=P(\Omega^\delta\Delta \mathbb B^n)\le C_1 (k_0!)^{-(n-1)/\delta},
   \end{equation}
   where $C_1=C_1(n)>0$, even though $\mathbb S^{n-1}\subset \partial D^\delta$.
\end{lem}
\begin{proof}
    (1) 
    For  any  point $x_0\in \Omega^\delta$, choose a point $x\in \Omega^\delta\setminus B_{1-2^{-(k'+1)}}$, where 
    $k'\ge k_0$ is the integer satisfying 
\begin{equation}\label{k'}
    1-2^{-k'}-|x_0|\ge 2^{-k'} 
\end{equation}
    Then for each curve $\gamma$ joining $x$ to $x_0$, there exists a point $y_{\gamma}\in\partial B_{1-2^{-k'}}\cap\gamma$. In virtue of \eqref{dense bo sep} and $E_k\subset (\Omega^\delta)^c$, we have 
$$
        \dist(y_{\gamma},\partial \Omega^\delta)\le 2\sqrt{n}(k'!)^{-1}.
$$
    Hence, \eqref{k'} implies  
    \begin{align}\label{not join}
\frac{\ell(\gamma[x,y_\gamma])}{\dist(y_{\gamma},\partial \Omega^\delta)}&\ge  \frac{|x-y_\gamma|}{\dist(y_{\gamma},\partial \Omega^\delta)}
\ge 2^{-k'-1}\dist(y_{\gamma},\partial \Omega^\delta)^{-1}\nonumber\\
&\ge n^{-\frac{1}{2}} 2^{-(k'+2)}k'! \ \to +\infty\  \text{ as }\,  k'\to \infty.
    \end{align}
    This implies that $\Omega^\delta$ is not a John domain.
    In particular,  $x$ is far away from $\overline{B}(y_\gamma,  \dist(y_{\gamma},\partial \Omega^\delta))$.  
    
    Moreover, note that \eqref{not join} gives
    $$  \frac{|x-y_\gamma|}{\dist(y_{\gamma},\partial \Omega^\delta)}\to \infty \text{ as }\  k'\to \infty.$$ 
    we claim that for every fixed $N>0$, if $k'\gg 1$ is sufficiently large, then
$$
\Omega^\delta\setminus B\!\left(y_\gamma,\,N\dist(y_\gamma,\partial\Omega^\delta)\right) \quad \text{ is path-connected}.
$$

Indeed, set
$$
R_\gamma:=N\dist(y_\gamma,\partial\Omega^\delta).
$$
Since $\dist(y_\gamma,\partial\Omega^\delta)\le 2\sqrt n\,(k'!)^{-1}$ and
$2^{-k'}k'!\to\infty$ as $k'\to\infty$, by taking $k'$ large we may assume
$R_\gamma<2^{-k'-3}$. It follows that the ball $B(y_\gamma,R_\gamma)$ is so small
that it can intersect only finitely many deleted balls, and in fact only those from the
$k'$-th generation.
 Hence
$$
K_\gamma:=\overline{B}(y_\gamma,R_\gamma)\cup
\bigcup\Bigl\{\overline{B}^\delta_{z,k}:\overline{B}^\delta_{z,k}\cap
\overline{B}(y_\gamma,R_\gamma)\neq\emptyset,\, k\ge k_0, \, z\in E_k\Bigr\}
$$
is a compact obstacle.

Now start from the segment $[x,x_0]$. Whenever this segment meets a deleted ball
$\overline{B}^\delta_{z,k}$ disjoint from $K_\gamma$, replace the crossed chord by an arc
on $\partial (2B^\delta_{z,k})$. Since $n\ge2$, this can always be done inside
$\Omega^\delta$. Furthermore,
$$
\sum_{k\ge k_0}\sum_{z\in E_k}\diam(2B^\delta_{z,k})
\le C(n)\sum_{k\ge k_0}2^{-k}(k!)^{\,n-1-1/\delta}<\infty,
$$
because $1/\delta>n-1$. Therefore the total length added by all detours is finite,
and after also detouring around the compact set $K_\gamma$ we obtain a rectifiable
curve $\widetilde\gamma\subset \Omega^\delta\setminus B(y_\gamma,R_\gamma)$ joining
$x$ to $x_0$. Consequently,
$$
x \text{ and } x_0 \text{ belong to the same path-component of }
\Omega^\delta\setminus B(y_\gamma,R_\gamma),
$$
and $\Omega^\delta$ does not satisfy the ball separation property either.
    
    (2)  By the construction of $\Omega^\delta$, 
\begin{equation}\label{partial-omega}
    \partial \Omega^\delta= \mathbb S^{n-1}\cup\left(\bigcup_{k=k_0}^{\infty}\bigcup_{x\in E_k} \partial B_{x,k}^\delta\right).
\end{equation}

A direct computation gives that, for any $\delta\in (0,n^{-2}]$,  
\begin{align}\label{inn bound}
   &\mathcal H^{n-1}(\partial \Omega^\delta)-\mathcal H^{n-1}(\mathbb S^{n-1})=\sum_{k=k_0}^\infty\sum_{x\in E_k} \mathcal H^{n-1}\left(\partial B_{x,k}^\delta\right)=C(n)\sum_{k=k_0}^\infty (k!)^{n-1} \frac{2^{-(n-1)k}}{(k!)^{(n-1)/\delta}}\notag\\
   &\qquad\qquad\le C(n) (k_0!)^{(n-1)(1-1/\delta)}\sum_{k=k_0}^\infty 2^{-(n-1)k}\leq C(n) (k_0!)^{-(n-1)/\delta}<\infty.
\end{align}
Similarly, for any $m\ge k_0+1$,  $D_m^\delta:=\bigcup_{k=k_0}^{m}\bigcup_{x\in E_k} \overline {B_{x,k}^\delta}$ has Lipschitz boundary so that  
$$\lim_{m\to \infty}|(\mathbb B^n\setminus D_m^\delta)\Delta \Omega^\delta|=0\quad \text{and}\quad P(\mathbb B^n\setminus D_m^\delta)=\mathcal H^{n-1}(\partial(\mathbb B^n\setminus D_m^\delta))<\mathcal{H}^{n-1}(\partial\Omega^\delta).$$ 
Thus by the lower semicontinuity of the perimeter (cf. \cite[Theorem 12.15]{M2012}), $\Omega^\delta$ is a set of finite perimeter.

Hence, according to Lemma \ref{Federer}, in order to obtain $\mathcal H^{n-1}(\partial\Omega^\delta\setminus\partial^*\Omega^\delta)=0$, it suffices to show that 
\begin{equation}\label{t in m}  \partial\Omega^\delta\subset\partial_M\Omega^\delta.
\end{equation}

Note that in the construction of $\Omega^\delta$, the deleted balls $\{\overline B_{x,k}^\delta: x\in E_k,\, k=k_0, 3, \cdots\}$ are chosen with so small radii  that, for any pair of integers $k,l\ge k_0$ and any pair of distinct points $x\in E_k$ and $y\in E_l$, 
\begin{equation}\label{disjoint bo}
    2\overline B_{x,k}\cap 2\overline B_{y,l}=\emptyset.
\end{equation}
Therefore, $\partial B_{x,k}^\delta\subset \partial_M \Omega$ holds for any $x\in E_k$  with $k\ge k_0$. 

We next show that $\mathbb S^{n-1}\subset \partial_M\Omega$.
Let $z\in \mathbb S^{n-1}$ and $0<r<1/4$. Owing to the geometry of $\Omega^\delta$, if $\overline B_{x,k}^\delta\cap B(z,r)\neq \emptyset$,  then $\overline B_{x,k}^\delta\subset B(z, 2r)$ and hence $2r>1-|x|=2^{-k}$. Let $k_r=\lceil -\log_2 r \rceil-1$, where $\lceil -\log_2 r \rceil$ denotes  the smallest integer larger than or equal to $-\log_2 r$. Since
$$2^{-k_r}\leq 2r\quad \text{and}\quad k_r\to +\infty\quad  \text{as } r\rightarrow 0^+,$$ 
and since there are at most $C(n)(k!)^{n-1}$-many balls $B_{x,\,k}$ of radius $2^{-k-2}(k!)^{-n}$ in question, 
we have
\begin{align}\label{bd 0 mea}
	&\frac{|D^\delta\cap B(z, r)|}{|B(z, r)|}\leq \frac{1}{|B(z,r)|}\sum_{k= k_r}^{\infty}\sum_{x\in E_k}| \overline B^\delta_{x,k}|\leq  \frac{1}{|B(z,r)|}\sum_{k= k_r}^{\infty}\sum_{x\in E_k}| B_{x,k}|\nonumber\\
    = &\frac{C(n)}{ r^n}\sum_{k= k_r}^{\infty}2^{-n(k+2)}(k!)^{n-1-n^2}
	\leq  \frac{C(n)k_r^{n-1-n^2}}{1-2^{-n}}\to 0^+\quad \text{as }r\to 0^+.
\end{align}
In particular, we have
\begin{align*}
  \frac{|\Omega^\delta\cap B(z,r)|}{|B(z,r)|}= \frac{|\mathbb B^n\cap B(z,r)|}{|B(z,r)|}-\frac{|D^\delta\cap B(z, r)|}{|B(z, r)|}\rightarrow \frac{1}{2} \quad \text{as }r\to 0^+,
\end{align*}
so that  $z \in \partial_{M}\Omega^\delta$. This implies, by the arbitrariness of $z$, that $\mathbb S^{n-1} \subset \partial_M\Omega^\delta$, hence concluding \eqref{t in m}. 

It remains to show \eqref{inn bound lem}. Note that, by the arbitrariness of $z\in \mathbb S^{n-1}$, it follows from \eqref{bd 0 mea} that $\mathbb S^{n-1}\subset (D^\delta)^{(0)}$. Thus, we obtain from  $\partial^* D^\delta\subset \partial_M D^\delta$ and $\partial_M D^\delta \cap (D^\delta)^{(0)}=\emptyset$ that
$$\partial^* D^\delta\subset \mathbb R^n\setminus \mathbb S^{n-1}.$$
This, coupled with $\partial D^\delta= \partial \Omega^\delta$, implies $\partial^* D^\delta\subset \partial \Omega^\delta\setminus \mathbb S^{n-1}$. Whence,  \eqref{inn bound} yields \eqref{inn bound lem}, and we  conclude the proof of the property $(2)$.
\end{proof}

\begin{prop}\label{tau value}
    There exist two constants $c_0=c_0(n)\in(0,1)$ and $0<\epsilon_0=\epsilon_0(n)\ll 1$, and a modulus of continuity $\mu:\mathbb R^+\to \mathbb R^+$ such that, for any $\delta\in (0, \epsilon_0)$,
    \begin{equation}\label{tau o est}
        c_0\mu(\delta)<\tau(\mathbb B^n)-\tau(\Omega^\delta)<\mu(\delta).
    \end{equation}
\end{prop}

\begin{proof} 
Let $\lambda_0=\lambda_0(n)>0$ and $a=a(n)>0$ be two constants such that 
\begin{equation}\label{lamb a}
    n\omega_n^{\frac{1}{n}}\lambda_0^{\frac{n-1}{n}}=3\tau(\mathbb B^n)\quad \text{and}\quad \lambda_0 a^{\frac{n}{n-1}}=|\mathbb B^n|/10.
\end{equation}
Let $\delta\in (0, \epsilon_0)$ small for some $\epsilon_0>0$ to be determined.

We first show that, given any $F\subset \Omega^\delta$ with $0<|F|\le |\Omega^\delta|/2$ and 
 $\mathcal{H}^{n-1}(\partial^*\Omega^\delta\cap \partial^* F)>0$, there exists a modulus of continuity $\mu:\mathbb R^+\to \mathbb R^+$ independent of $F$ such that
\begin{equation}\label{tau disk}
(\tau(\mathbb B^n)-\mu(\delta))\mathcal{H}^{n-1}(\partial^*\Omega^\delta\cap \partial^*F)\leq P(F)\quad \text{ for any }0<\delta\le \epsilon_0.
\end{equation}
Once this holds, the second inequality in \eqref{tau o est} follows.

\medskip
\noindent{\bf Step 1: Show \eqref{tau disk} in the case $\mathcal{H}^{n-1}(\mathbb S^{n-1}\cap \partial^*F)> a$.} Recall the definition of $a$ in \eqref{lamb a}. 
Observe that 
 $\mathcal H^{n-1}(\partial \Omega^\delta\setminus \partial^*\Omega^\delta)=0$ and \eqref{inn bound} imply
 \begin{equation}\label{mu control 1}
     \mathcal{H}^{n-1}(\partial^*\Omega^\delta\cap \mathbb B^n)=  \mathcal H^{n-1}(\partial \Omega^\delta)-\mathcal H^{n-1}(\mathbb S^{n-1})\leq  C_1 (k_0!)^{-(n-1)/\delta}, 
 \end{equation}
where $C_1$ is the constant defined in Lemma \ref{ch3 fin per}.  Then, by letting 
$$\mu_1(\delta)=C_1 a^{-1} (k_0!)^{-(n-1)/\delta} \quad \text{ for } \  \delta>0,$$ 
which depends only on $n$, we get
\begin{equation}\label{del 2}
    \mathcal{H}^{n-1}(\partial^*\Omega^\delta\cap \partial^*F)-\mathcal{H}^{n-1}(\mathbb S^{n-1}\cap \partial^*F)\le \mathcal{H}^{n-1}(\partial^*\Omega^\delta \cap \mathbb B^n)< a\mu_1(\delta).
\end{equation}

Thanks to $\mathcal{H}^{n-1}(\mathbb S^{n-1}\cap \partial^* F)>a=a(n)$, the definition of $\tau(\mathbb B^n)$ and $|F|<\frac{1}{2}|\mathbb B^n|$ yield $$a<\mathcal H^{n-1}(\mathbb S^{n-1}\cap \partial^*F)\le\tau(\mathbb B^n)^{-1}P(F).$$ 
This together with  \eqref{del 2} implies that 
\begin{align*}
\mathcal{H}^{n-1}(\partial^*\Omega^\delta\cap \partial^*F)&<\mathcal H^{n-1}(\mathbb S^{n-1}\cap \partial^*F)+a\mu_1(\delta)\le  \tau(\mathbb B)^{-1}(1+\mu_1(\delta))P(F).
\end{align*} 
Thus, we conclude  \eqref{tau disk} with 
\begin{equation}\label{case1}
    \mu(\delta)=\tau(\mathbb B^n)\mu_1(\delta)\ge \tau(\mathbb B^n)\mu_1(\delta) /(1+\mu_1(\delta))
\end{equation}
where $\mu$ depends only on $n$.

In what follows, we assume that
\begin{equation}\label{case2}
    \mathcal{H}^{n-1}(\mathbb S^{n-1}\cap \partial^*F)\le a=a(n).
\end{equation}

\medskip

\noindent{\bf Step 2: Replace $F$ by a regular set $U$ under \eqref{case2}.} Let $b>0$ small (and particularly independent of $\delta$). By applying the argument in \cite[Step 1, Proof of Lemma 2.4]{SZ2024}, there exists a smooth open set $\widetilde F\subset \mathbb R^n$ satisfying
\begin{equation}\label{replace 1}
   P(\widetilde F)\le (1+b)P(F),\quad P(\widetilde F;\overline{\Omega}^\delta)\le (1+b)P(F;\Omega^\delta)  
\end{equation}
and
\begin{equation}\label{replace 2}
  \mathcal{H}^{n-1}(\partial^*\Omega^\delta\cap \partial^*F)\le (1+b)\mathcal{H}^{n-1}(\partial^*\Omega^\delta\cap \widetilde F).
\end{equation}
In particular, $\mathcal{H}^{n-1}(\partial^*\Omega^\delta\cap \widetilde F)\neq 0. $

Next, we consider a set of finite perimeter $ U\subset \mathbb R^n$ minimizing the Plateau problem:
$$\inf\{P(G): G\setminus \Omega^\delta = \widetilde F\setminus\Omega^\delta\}.$$
From \cite[Proposition 12.29]{M2012} and \cite[Example 16.13]{M2012},  it follows that $U$ exists and is a perimeter constraint minimizer. As  $U\setminus \Omega^\delta= \widetilde F\setminus \Omega^\delta$ coming from \cite[Proof of Proposition 12.29]{M2012}, we obtain that \begin{equation}\label{minimum}
    P(U)\le P(\widetilde F),\quad  P(U;\overline{\Omega}^\delta)\le P(\widetilde F; \overline{\Omega}^\delta)\quad \text{and}\quad  \mathcal{H}^{n-1}(\partial^*\Omega^\delta\cap U)=\mathcal{H}^{n-1}(\partial^*\Omega^\delta\cap \widetilde F)\neq 0,
\end{equation}
where the second inequality is given by \cite[Exercise 12.16]{M2012}.
Moreover, \cite[Theorem 16.14]{M2012} tells that there is a positive constant $c(n)$ such that
\begin{equation}\label{desity estimate mini}
c(n)\rho^{n-1}\le P(U;B(y,\rho))\le n\omega_n\rho^{n-1} \quad \text{ for any  $B(y, 2\rho)\subset\subset \Omega^\delta$ with $y\in \Omega^\delta\cap \partial U$}.
\end{equation}

We now consider two subcases 
$$|U\cap \Omega^\delta|> \lambda_0 a^{\frac{n}{n-1}}\quad \text{and}\quad |U\cap \Omega^\delta|\le \lambda_0 a^{\frac{n}{n-1}}.$$

\medskip
\noindent{\bf Step 3: Show \eqref{tau disk} when $|U\cap \Omega^\delta|> \lambda_0 a^{\frac{n}{n-1}}$ and \eqref{case2}.}
Recall that $\lambda_0=\lambda_0(n)$ and $a=a(n)$ from \eqref{lamb a}. 
Under the assumption \eqref{case2}, it follows from \eqref{del 2} that there exists $0<\epsilon_1=\epsilon_1(n)\ll a$ such that, whenever $0<\delta<\epsilon_1$, 
\begin{equation}\label{replace lhs}
    \mathcal{H}^{n-1}(\partial^*\Omega^\delta\cap \partial^* F) \leq a+\mu_1(\delta) \leq 2a.
\end{equation}
Since $|U\cap \Omega^\delta|> \lambda_0 a^{\frac{n}{n-1}}$, the isoperimetric inequality \eqref{iso 2} and the first equality in \eqref{lamb a} imply
$$P(U)\ge n \omega_n^{\frac{1}{n}}|U\cap \Omega^\delta|^{\frac{n-1}{n}}> 2a \tau(\mathbb B^n),$$
which together with \eqref{replace lhs} gives
$$P(U)> \tau(\mathbb B^n) \mathcal{H}^{n-1}(\partial^*\Omega^\delta\cap \partial^* F).$$
Thus, combining this with \eqref{replace 1} and
\eqref{minimum}, we conclude that
$$\tau(\mathbb B^n) \mathcal{H}^{n-1}(\partial^*\Omega^\delta\cap \partial^* F) < P(U)\leq P(\widetilde F)\leq (1+b) P(F).$$
Hence, by eventually sending $b\to 0^+$, we establish \eqref{tau disk} with 
\begin{equation}\label{subcase1}
    \epsilon_0=\epsilon_1\quad \text{and}\quad \mu(\delta)\text{ taken to be any modulus of continuity.}
\end{equation}

\noindent{\bf Step 4: Show \eqref{tau disk} when $|U\cap \Omega^\delta|\le  \lambda_0 a^{\frac{n}{n-1}}$ and \eqref{case2}.} 
By Lemma \ref{Maggi}, 
$$P(F)-P(F;\Omega^\delta)=\mathcal{H}^{n-1} (\partial^*\Omega^\delta\cap \partial^*F).$$
Thus, via \eqref{minimum} and \eqref{replace 1}, we have
\begin{equation}\label{u le f}
    P(U;\overline{\Omega}^\delta)\le P(\widetilde F; \overline{\Omega}^\delta)\le (1+b)P(F;\Omega^\delta)=(1+b)(P(F)-\mathcal{H}^{n-1} (\partial^*\Omega^\delta\cap \partial^*F)).
\end{equation}
Furthermore, combining \eqref{replace 2} and \eqref{minimum}, we get 
$$(1+b)^{-1}\mathcal{H}^{n-1} (\partial^*\Omega^\delta\cap \partial^*F)\le \mathcal H^{n-1}(\partial^*\Omega^\delta\cap \widetilde F)=\mathcal H^{n-1}(\partial^*\Omega^\delta\cap   U).$$
This coupled with \eqref{u le f}, tells that 
\begin{multline}
\label{tau disk 0}
    \left[(1+b)^{-2}(\tau(\mathbb B^n)-1-\hat \mu(\delta))+1\right]\mathcal{H}^{n-1} (\partial^*\Omega^\delta\cap \partial^*F)\\
    \le (1+b)^{-1}  P(U;\overline{\Omega}^\delta) +  \mathcal{H}^{n-1} (\partial^*\Omega^\delta\cap \partial^*F)\le P(F).
\end{multline}

We next show the existence of a modulus of continuity $\hat \mu:\mathbb R^+\to\mathbb R^+$, independent of $b$, satisfying
\begin{equation}\label{regu epsi2}
	(\tau(\mathbb B^n)-1-\hat \mu(\delta)) \mathcal{H}^{n-1}(\partial^*\Omega^\delta\cap U) \le P(U; \overline{\Omega}^\delta)\quad \text{for any  }0<\delta\le  \epsilon_0.
\end{equation}
Once this holds, then by letting
\begin{equation}\label{final-mu}
    \mu(\delta)=\frac{b^2+2b}{(1+b)^2}(\tau(\mathbb B^n)-1)+\frac{\hat \mu(\delta)}{(1+b)^2},
\end{equation}
and eventually sending $b\to 0^+$, \eqref{tau disk 0} directly implies \eqref{tau disk} with $\mu(\delta)=\hat\mu(\delta)$.

\noindent{\bf Step 4.1: Estimate on $\mathcal{H}^{n-1}(\partial^*\Omega^\delta\cap U)$.}
In order to estimate $\mathcal{H}^{n-1}(\partial^*\Omega^\delta\cap U)$, we divide it into four parts.
 For any set $G\subset\mathbb R^n$, denote  $$I_\epsilon(G):=\{z\in \mathbb R^n: \dist(z, G)< \epsilon\}.$$
 Let $a_k=(k!)^{-1/\delta} 2^{-k-9},\, k\ge k_0,$ and set
 \begin{align}\label{classify}
    & \mathscr{A}:= \left\{(x,k):I_{a_k}(\partial B_{x,k})\subset \inter(U), \,x\in E_k,\, k\geq k_0\right\},\nonumber\\
    & \mathscr{K}:=\left\{(x,k):I_{a_k}(\partial B_{x,k})\cap \partial U\neq \emptyset,\, x\in E_k, \, k\geq k_0\right\},\\
    & \mathscr{D}:=\left\{(x,k):I_{a_k}(\partial B_{x, k})\cap \overline{U}=\emptyset,\, x\in E_k,\, k\geq k_0\right\}\nonumber.
 \end{align}
We further set 
\begin{equation}\label{single compo}
    \widetilde U:=U \setminus \bigcup_{(x,k)\in \mathscr{D}}B_{x,k}\quad\text{and}\quad U_{x,k}:=U \cap B_{x,k}\quad \text{for any }(x,k)\in \mathscr{D}.
\end{equation}

For any $k\geq k_0$ and $x\in E_k$, since $I_{a_k}(\partial B_{x,k})$ is connected, it follows that if $I_{a_k}(\partial B_{x,k})\cap \overline U\not=\emptyset$,  then either $(x, k)\in \mathscr A$ or $(x, k)\in \mathscr K$. Therefore, 
$$\mathscr{A}\cup \mathscr K\cup \mathscr D=\{(x, k): x\in E_k,\ k\geq k_0\}.$$
In addition, since all the balls $2B_{x,k}$ are disjoint from $\mathbb S^{n-1}$, it follows from \eqref{disjoint bo} that
\begin{align}\label{lhs esti1}
     \mathcal{H}^{n-1}(\partial^*\Omega^\delta\cap U)\le & \mathcal{H}^{n-1}(\mathbb S^{n-1}\cap \widetilde U)+\sum_{(x,k)\in \mathscr{A}}\mathcal{H}^{n-1}\left(\partial B_{x,k}^\delta\right)\nonumber\\
     &\quad + \sum_{(x,k)\in \mathscr{K}}\mathcal{H}^{n-1}\left(\partial B_{x,k}^\delta\right)+ \sum_{(x,k)\in \mathscr{D}}\mathcal{H}^{n-1}\left(\partial B_{x,k}^\delta\cap  U\right)\nonumber\\
      =: & T_1+T_2+T_3+T_4. 
\end{align}
In what follows we estimate these four terms separately. 

\noindent{\bf Step 4.2: Estimate on $T_1$.}
Applying the fact $U\setminus \Omega^\delta=\widetilde F\setminus \Omega^\delta$, it follows from the definition of $\widetilde U$ in \eqref{single compo} that
$$\widetilde U\setminus \mathbb B^n= U\setminus \mathbb B^n=\widetilde F\setminus \mathbb B^n.$$
Due to the openness of $\widetilde F$, for any $z\in \mathbb S^{n-1}	\cap U=\mathbb S^{n-1}\cap \widetilde F$,  when the radius $\rho\ll 1$, 
  $$B(z,\rho)\cap (\widetilde U\setminus \mathbb B^n)=B(z,\rho)\cap (\widetilde F\setminus \mathbb B^n)=B(z,\rho)\setminus \mathbb B^n.$$ 
Thus, by letting $\rho\to 0^+$, we have 
$$\lim_{\rho\to 0^+}\frac{|B(z,\rho)\cap (\widetilde U\setminus \mathbb B^n)|}{|B(z,\rho)|}=\frac{1}{2}.$$ 
Analogously, recalling \eqref{single compo}, we also have 
$$\lim_{\rho\to 0^+}\frac{|B(z,\rho)\cap (U_{x,k}\cap  B_{x,k}^\delta)|}{|B(z,\rho)|}=\frac{1}{2}$$
for any $(x,k)\in \mathscr{D}$ and $z\in \partial B_{x,k}^\delta\cap U$. Therefore,
\begin{equation}\label{12 include}
    \mathbb S^{n-1}\cap \widetilde U\subset (\widetilde U\setminus \mathbb B^n)^{(\frac{1}{2})}\cap   \mathbb S^{n-1}\quad\text{and}\quad \partial B_{x,k}^\delta\cap U\subset (U_{x,k}\cap  B_{x,k}^\delta)^{(\frac{1}{2})}\cap (B_{x,k}^\delta)^c;
\end{equation}
recall \eqref{density set}.
By Federer's Theorem and applying Lemma \ref{Maggi}  to both $\widetilde U\setminus\mathbb B^n$ and $\Omega^\delta\cap \widetilde U$, \eqref{12 include}  yields
\begin{align}\label{B contr1}
    \mathcal{H}^{n-1}(  \mathbb S^{n-1}\cap \widetilde U)\le & P(\widetilde U\setminus \mathbb B^n;\mathbb S^{n-1}) \le  \mathcal{H}^{n-1}(\{\nu_{\widetilde U}=-\nu_{\mathbb B^n}\})+\mathcal{H}^{n-1}(\widetilde U^{(1)}\cap \mathbb S^{n-1})\nonumber\\
   \le & \mathcal{H}^{n-1}(\{\nu_{\widetilde U}=-\nu_{\mathbb B^n}\})+\mathcal{H}^{n-1}( \partial^* (\Omega^\delta\cap \widetilde U)\cap \mathbb S^{n-1}). 
\end{align}
Furthermore, due to $|\Omega^\delta\cap \widetilde U|\le \lambda_0 a^{\frac{n}{n-1}}$ and \eqref{lamb a}, the definition of $\tau(\mathbb B^n)$ yields 
\begin{equation}\label{B contr22}
    \mathcal{H}^{n-1}( \partial^* (\Omega^\delta\cap \widetilde U)\cap \mathbb S^{n-1})\le (\tau(\mathbb B^n)-1)^{-1}P(\Omega^\delta\cap \widetilde U; \mathbb B^n).
\end{equation}

We let 
$$
     B_{half}:=\mathbb B^n\cap \{x=(x_1,x_2,\cdots,x_n)\in \mathbb R^n:x_1<0\}.
$$
By the definition of $\tau(\mathbb B^n)$ and applying Lemma \ref{Maggi} to $ B_{half}$,
 \begin{equation}\label{B half}
     (\tau(\mathbb B^n)-1)^{-1}\ge \left( \frac{P(B_{half})}{\mathcal{H}^{n-1}(\mathbb S^{n-1}\cap \partial^* B_{half})}-1\right)^{-1} =\frac{\mathcal{H}^{n-1}(\mathbb S^{n-1}\cap \partial^* B_{half})}{P(B_{half};\mathbb B^n)}>1.
 \end{equation}
Therefore, we obtain from \eqref{B contr1}, \eqref{B contr22} 
and Lemma \ref{Maggi} that
\begin{align}\label{B contr1.5}
    & \mathcal{H}^{n-1}(\mathbb S^{n-1}\cap \widetilde U)  \le P(\widetilde U; \mathbb S^{n-1})+(\tau(\mathbb B^n)-1)^{-1}P(\Omega^\delta\cap \widetilde U; \mathbb B^n)\nonumber\\
   &\quad\quad \le   (\tau(\mathbb B^n)-1)^{-1}P(\widetilde U ;\mathbb S^{n-1}) + (\tau(\mathbb B^n)-1)^{-1}( P(\widetilde U;\overline{\Omega}^\delta\cap \mathbb B^n)+P(\Omega^\delta; \mathbb B^n\cap \widetilde U^{(1)})).
\end{align}
Moreover, \eqref{partial-omega}, the definition of $\widetilde U$ in \eqref{single compo} together with the definition of $T_2, T_3$ in \eqref{lhs esti1} give
\begin{equation}\label{A and K}
    P(\Omega^\delta;\mathbb B^n\cap \widetilde U^{(1)})\leq T_2+T_3.
\end{equation}
Thus, \eqref{B contr1.5} yields 
\begin{equation}\label{B contr2}
    \mathcal{H}^{n-1}(\mathbb S^{n-1}\cap \widetilde U)\le (\tau(\mathbb B^n)-1)^{-1}(P(\widetilde U; \overline{\Omega}^\delta )+ T_2+T_3).
\end{equation}

\noindent{\bf Step 4.3: Estimate on $T_2$.}
Let 
$$k_U:=\inf\{k\ge k_0:\text{ there exists } (y,k)\in\mathscr{A}\}\geq k_0,$$ 
and choose $x_U\in E_{k_U}$ for some $(x_U,k_U)\in \mathscr{A}$. Then, when $0<\delta<(2n+2)^{-2}$, since there are at most $C(n)(k!)^{n-1}$-many points in $E_k$ and $n+1-1/\delta<0$, 
\begin{align}\label{estimate-A-0}
        T_2 & \le \sum_{k\ge k_U}\sum_{x\in E_k}\mathcal{H}^{n-1}\left(\partial B_{x,k}^\delta\right) =\sum_{k\ge k_U}\sum_{x\in E_k} n\omega_n \left(\frac{2^{-k}}{(k!)^{1/\delta}}\right)^{n-1} \le C(n)(k_U!)^{(n-1)(1-1/\delta)}\sum_{k\ge k_U}2^{-(n-1)k}\notag\\
    &
\le C(n)(k_U!)^{(n-1)(n+1-1/\delta)}\frac{2^{-(k_U-2)(n-1)}}{(k_U!)^{n(n-1)}}=C(n)(k_U!)^{(n-1)(n+1-1/\delta)}|B_{x_U, k_U}|^{\frac{n-1}{n}}\notag\\
  & \leq C(n) (k_0!)^{(n-1)(n+1-1/\delta)}|B_{x_U, k_U}|^{\frac{n-1}{n}}\le    C(n) (k_0!)^{(n-1)(n+1-1/\delta)}|B_{x_U,k_U}\cup (\widetilde U\cap \Omega^\delta)|^{\frac{n-1}{n}}.
\end{align}

 Write $U_0=B_{x_U,k_U}\cup (\widetilde U\cap \Omega^\delta)$ for brevity. As $|\Omega^\delta\cap \widetilde U|\le \lambda_0 a^{\frac{n}{n-1}}$, by the second equality in \eqref{lamb a}, it holds that whenever $\delta<\epsilon_2$ for some $0<\epsilon_2=\epsilon_2(n)<\epsilon_1$,
 $$|U_0|\le|U\cap \Omega^\delta|+|2B_{x_U,k_U}|<|\mathbb B^n|/2<|\mathbb B^n\setminus U_0|.$$  
  Hence, by the relative isoperimetric inequality \eqref{rela iso}, we have
 $$C_{re}|U_0|^{\frac{n-1}{n}}\le P(U_0; \mathbb B^n),$$
 where $C_{re}>0$ is as in \eqref{rela iso}. 
 This, coupled with \eqref{estimate-A-0}, implies that there exist a constant $0<\epsilon_3=\epsilon_3(n)<\min\{\epsilon_2, (2n+2)^{-2}\}$ and a modulus of continuity $\mu_2:\mathbb R^+\to \mathbb R^+,$ such that, whenever $0<\delta<\epsilon_3$,
 \begin{equation}\label{estimate-A-1}
     T_2 \le  C(n)(k_0!)^{-(n-1)/\delta}P(U_0; \mathbb B^n)=  \mu_2(\delta)P(U_0; \mathbb B^n)\quad \text{and}\quad \mu_2(\delta)<1/2.
 \end{equation}

 Since $(x_U,k_U)\in\mathscr{A}$, we have $\partial B_{x_U,k_U}\subset (\widetilde U\cap \Omega^\delta)^{(1)}$. Then, by applying Lemma \ref{Maggi} to both $U_0$ and $\widetilde U\cap \Omega^\delta$, we obtain 
 $$P(U_0; \mathbb B^n)= P(\widetilde U\cap \Omega^\delta;\mathbb B^n)\le P(\widetilde U; \overline{\Omega}^\delta)+P(\Omega^\delta;\mathbb B^n\cap \widetilde U^{(1)}).$$
Whence, using \eqref{estimate-A-1} and \eqref{A and K}, we have  
\begin{align}
T_2\le \mu_2(\delta)\left(P(\widetilde U; \overline{\Omega}^\delta)+P(\Omega^\delta;\mathbb B^n\cap \widetilde U^{(1)})\right)\le \mu_2(\delta)\left(P(\widetilde U; \overline{\Omega}^\delta)+T_2+T_3\right).\nonumber
\end{align}
Absorbing the $\mu_2(\delta)T_2$-term on the right hand side via the left hand side, we obtain 
\begin{equation}\label{A contr}
	T_2\leq (1-\mu_2(\delta))^{-1}\mu_2(\delta)  (P(\widetilde U; \overline{\Omega}^\delta)+ T_3).
\end{equation} 

\noindent{\bf Step 4.4: Estimate on $T_3$.}
 For any $(x, k)\in \mathscr K$, there exists a point $z_x\in \partial U\cap I_{a_k}(\partial B_{x,k})$ so that 
\begin{equation}\label{ball inclu}
	B\left(z_x, \frac{2^{-k-3}}{(k!)^n}\right)\subset\subset 2B_{x,k}\setminus \overline B_{x,k}^\delta\subset \Omega^\delta
\end{equation}
holds by the definition of $\mathscr K$ in \eqref{classify} whenever $\delta<\epsilon_3$. Then it follows from \eqref{desity estimate mini} with $y=z_x$ and $\rho=\frac{2^{-k-4}}{(k!)^n}$ that
$$
\mathcal{H}^{n-1}\left(\partial B_{x,k}^\delta\right) \le n\omega_{n}2^{-k(n-1)}(k!)^{-(n-1)/\delta}\le  \frac{n\omega_n}{c(n)} 2^{4n-4} (k!)^{(n-1)(n-1/\delta)}P\left(U;B\left(z_x, \frac{2^{-k-4}}{(k!)^n}\right)\right).
$$
Since $k!\geq k_0!$ and $n-1/\delta<0$, it then follows from \eqref{ball inclu} that whenever $0<\delta<\epsilon_3$,
\begin{equation}\label{cover}
	\mathcal{H}^{n-1}\left(\partial B_{x,k}^\delta\right)\le \mu_3(\delta) P\left( U;2B_{x,k}\setminus \overline B_{x,k}^\delta\right)\quad \text{for any }(x,k)\in \mathscr{K},
\end{equation}
where $\mu_3(\delta)=\frac{n\omega_n}{c(n)} 2^{4n-4} (k_0!)^{(n-1)(n-1/\delta)}$. Due to \eqref{disjoint bo}, \eqref{cover} further implies
\begin{equation}\label{K contr}
	T_3\le \mu_3(\delta)\sum_{(x,k)\in \mathscr{K}}P\left( U; 2B_{x,k}\setminus B_{x,k}^\delta\right)\le \mu_3(\delta) P\left( U; \overline{\Omega}^\delta\right)\quad \text{whenever }0<\delta<\epsilon_3.
\end{equation}

\noindent{\bf Step 4.5: Estimate on $T_4$.} Fix $(x,k)\in \mathscr D$. Using Federer's theorem and the second inclusion in \eqref{12 include}, we have $$\mathcal{H}^{n-1}(\partial B_{x,k}^\delta\cap U)\le P(U_{x,k}\cap B_{x,k}^\delta; (B_{x,k}^\delta)^c).$$ Thus,
 observing that $P(U_{x,k}\cap B_{x,k}^\delta;B_{x,k}^\delta)=P(U_{x,k};B_{x,k}^\delta)$ follows from \cite[Example 12.16]{M2012}, we obtain that
\begin{align}\label{D contr1}
 \mathcal{H}^{n-1}(\partial B_{x,k}^\delta\cap U)&\le P(U_{x,k}\cap B_{x,k}^\delta; (B _{x,k}^\delta)^c)\nonumber\\
  &=P(U_{x,k}; (B_{x,k}^\delta)^c)+P(U_{x,k}\cap B_{x,k}^\delta)-P(U_{x,k}).
\end{align}
Furthermore, the isoperimetric inequality \eqref{iso 2} tells that 
\begin{equation}\label{isocompare}
    P(U_{x,k}\cap B_{x,k}^\delta)-P(U_{x,k})= P(B_{x,k}^\delta)-P(U_{x,k}\cup B_{x,k}^\delta)\le 0.
\end{equation}
Then since $(x,k)\in \mathscr{D}$, together with \eqref{single compo}, gives
$$P(U_{x,k}; B_{x,k}\setminus B_{x,k}^\delta))= P(U\setminus \widetilde U; B_{x,k}\setminus B_{x,k}^\delta),$$ it follows from \eqref{D contr1} and \eqref{isocompare} that 
$$
    \mathcal{H}^{n-1}(\partial B_{x,k}^\delta\cap U) \le P(U_{x,k}; (B_{x,k}^\delta)^c)= P(U\setminus \widetilde U; B_{x,k}\setminus B_{x,k}^\delta)\quad \text{for any }(x,k)\in \mathscr{D}.
$$
Therefore, by \eqref{disjoint bo} and the fact that $B_{x,k}\setminus B_{x,k}^\delta$ are  contained in $\overline{\Omega}^\delta$, we obtain
\begin{equation}\label{D contr2}
T_4\le \sum_{(x,k)\in \mathscr{D}}P(U\setminus \widetilde U; B_{x,k}\setminus B_{x,k}^\delta) \leq P(U\setminus \widetilde U; \overline{\Omega}^\delta).
\end{equation}

\noindent{\bf Step 4.6: Conclude \eqref{tau disk} when $|\Omega^\delta\cap \widetilde U|\le \lambda_0a^{\frac{n}{n-1}}$ and \eqref{case2}.}
As a consequence, by gathering \eqref{lhs esti1}, \eqref{A contr}, \eqref{K contr}, \eqref{B contr2} and \eqref{D contr2}, we have that whenever $0<\delta<\epsilon_3$,
\begin{equation}\label{conclude}
   \mathcal{H}^{n-1}(\partial^*\Omega^\delta\cap U)\le  \left((\tau(\mathbb B^n)-1)^{-1}+\omega_1(\delta)\right)P(\widetilde U; \overline{\Omega}^\delta)+ P(U\setminus \widetilde U; \overline{\Omega}^\delta)+\omega_2(\delta) P(U; \overline{\Omega}^\delta),
\end{equation}
where $\omega_1$ and $\omega_2$ are 
moduli of continuity given by
$$\omega_1(\delta)=\frac{\tau(\mathbb B^n)\mu_2(\delta)}{(\tau(\mathbb B^n)-1)(1-\mu_2(\delta))}\quad \text{and}\quad \omega_2(\delta)=\frac{\tau(\mathbb B^n) \mu_3(\delta)}{(\tau(\mathbb B^n)-1)(1-\mu_2(\delta))}.$$
Recall \eqref{single compo} gives $U\setminus\widetilde U\subset\bigcup_{(x,k)\in\mathscr{D}} B_{x,k}$. Then since \eqref{bd 0 mea} yields $\mathbb S^{n-1}\subset \left(\bigcup_{(x,k)\in\mathscr{D}}B_{x,k}\right)^{(0)}$, we have $$\mathbb R^n\setminus\bigcup_{(x,k)\in\mathscr{D}}\overline{B}_{x,k}\subset\Big(\bigcup_{(x,k)\in\mathscr{D}}B_{x,k}\Big)^{(0)}\subset(U\setminus\widetilde U)^{(0)}.$$
Hence, due to the inclusion 
$$cl(\widetilde U)\subset\mathbb R^n\setminus(\bigcup_{(x,k)\in \mathscr{D}}I_{a_k}(B_{x,k}))$$
following from the definition of $\mathscr{D}$ in \eqref{classify}, we obtain
\begin{equation}
    \widetilde U^{(1)}\cup \partial_M \widetilde U\subset (U\setminus\widetilde U)^{(0)},\quad (U\setminus\widetilde U)^{(1)}\cup \partial_M (U\setminus\widetilde U)\subset \widetilde U^{(0)} \quad \text{and}\quad\partial_M \widetilde U\cap \partial_M (U\setminus \widetilde U)=\emptyset.\nonumber
\end{equation}
Thus, by Lemma \ref{Federer}, we further obtain 
\begin{equation}
    P(\widetilde U;\overline{\Omega}^\delta)=P(\widetilde U;\overline{\Omega}^\delta\cap (U\setminus \widetilde U)^{(0)}), \quad P(U\setminus \widetilde U;\overline{\Omega}^\delta)=P(U\setminus \widetilde U;\overline{\Omega}^\delta\cap U^{(0)})\nonumber\\
\end{equation}
and 
$$\mathcal{H}^{n-1}(\{\nu_{\widetilde U}=\nu_{U\setminus\widetilde U}\})=0.$$
This implies, via Lemma \ref{Maggi} applied to $\widetilde U\cup (U\setminus \widetilde U)$,
that 
$$P(\widetilde U; \overline{\Omega}^\delta)+P(U\setminus \widetilde U; \overline{\Omega}^\delta)=P(U; \overline{\Omega}^\delta).$$
Then since \eqref{B half} implies $(\tau(\mathbb B^n)-1)^{-1}+\omega_1(\delta)>1$, we conclude from \eqref{conclude} that
$$ \mathcal{H}^{n-1}(\partial^*\Omega^\delta\cap U)\le \left((\tau(\mathbb B^n)-1)^{-1}+\omega(\delta)\right) P(U; \overline{\Omega}^\delta), \quad \omega(\delta)=\omega_1(\delta)+\omega_2(\delta).$$
Therefore, recalling the form of $\mu_2$ and $\mu_3$ in \eqref{estimate-A-1} and \eqref{cover}, respectively, we find that there exists a constant $C_2=C_2(n)>0$ such that \eqref{regu epsi2} holds with the modulus of continuity
\begin{equation}\label{subcase2 mu}
    \hat\mu(\delta)=(k_0!)^{-(n-1)/\delta}C_2 \ge \frac{(\tau(\mathbb B^n)-1)^2 \omega(\delta)}{1+(\tau(\mathbb B^n)-1)\omega(\delta)}\quad \text{and} \quad \epsilon_0=\epsilon_3.
\end{equation}
Hence, by letting $b\to 0^+$ and recalling the definition of $\mu$ in \eqref{final-mu}, we conclude \eqref{tau disk} with the modulus of continuity $\mu(\delta)=C_2 (k_0!)^{-(n-1)/\delta}$. 
\medskip

\noindent{\bf Step 5: An upper estimate on $\tau(\mathbb B^n)-\tau(\Omega^\dz)$.}
According to \eqref{case1}, \eqref{subcase1} and \eqref{subcase2 mu}, we conclude that 
there exists a constant $C_3=C_3(n)>0$ such that
\eqref{tau disk} holds with $\epsilon_0$, given in \eqref{subcase2 mu}, and 
\begin{equation}\label{mu}
    \mu(\delta)=C_3(k_0!)^{-(n-1)/\delta}.
\end{equation}
This yields the upper bound on $\tau(\mathbb B^n)-\tau(\Omega^\dz)$ in \eqref{tau o est}.

\medskip
\noindent{\bf Step 6: A lower bound on $\tau(\mathbb B^n)-\tau(\Omega^\dz)$.}
Finally, we show the first inequality in \eqref{tau o est}. To this end, given $\delta\in (0,\epsilon_0)$, we  estimate  
$$\mathcal{H}^{n-1}(\partial^*\Omega^\delta\cap \partial^*(E_{\varphi,\vartheta}\cap \Omega^\delta))$$
 from above and below, respectively, 
where $E_{\varphi,\vartheta}$ is the set defined in Theorem \ref{cite cianchi}.
 
By Lemma \ref{k0} and the openness of $E_{\varphi,\vartheta}$, there exists a point $\hat x\in E_{k_0}$ such that
$$\overline{B}_{\hat x,k_0}^\delta\subset\subset E_{\varphi,\vartheta}\subset E_{\varphi,\vartheta}^{(1)}.$$
 Thus, since $\partial B_{\hat x,k_0}^\delta\cup\mathbb S^{n-1}\subset\partial^*\Omega^\delta$, applying Lemma \ref{Maggi} to $E_{\varphi,\vartheta}\cap\Omega^\delta$ yields that 
\begin{align}\label{denomin}
  &   \mathcal{H}^{n-1}(\partial^*\Omega^\delta\cap \partial^*(E_{\varphi,\vartheta}\cap \Omega^\delta))=\mathcal{H}^{n-1}(\partial^*\Omega^\delta\cap \partial^* E_{\varphi,\vartheta})+\mathcal{H}^{n-1}(\partial^*\Omega^\delta\cap  E_{\varphi,\vartheta}^{(1)})\nonumber\\
 &\qquad\ge \mathcal{H}^{n-1}(\partial^*\Omega^\delta\cap \partial^* E_{\varphi,\vartheta}) + \mathcal{H}^{n-1}(\partial B_{\hat x,k_0}^\delta)\ge  \mathcal{H}^{n-1}(\mathbb S^{n-1}\cap \partial^*E_{\varphi,\vartheta})+ \mathcal{H}^{n-1}(\partial B_{\hat x,k_0}^\delta),
\end{align} 
leading to the desired lower bound.

On the other hand, combining \eqref{volumn equi}, \eqref{disjoint bo} and \eqref{volume est layer}, we obtain
\begin{align}\label{halfmoon vol}
     |E_{\varphi,\vartheta}\cap \Omega^\delta| &=|E_{\varphi,\vartheta}|-\sum_{k\ge k_0}\sum_{x\in E_k}|\overline{B}_{x,k}^\delta\cap E_{\varphi,\vartheta}| \le |E_{\varphi,\vartheta}|-\frac{1}{2}\sum_{k\ge k_0}\sum_{x\in E_k}|\overline{B}_{x,k}^\delta| \nonumber\\
     & = \frac{1}{2}|\mathbb B^n| -\frac{1}{2}\sum_{k\ge k_0}\sum_{x\in E_k}|\overline{B}_{x,k}^\delta|=\frac{1}{2}|\Omega^\delta|.
 \end{align}
Furthermore, due to $\Omega^\delta\subset \mathbb B^n$, one has $(\Omega^\delta)^{(1)}\subset\mathbb B^n$ . Hence, by \eqref{halfmoon vol}, the definition of $\tau(\Omega^\delta)$ and applying Lemma \ref{Maggi} to $E_{\varphi,\vartheta}\cap \Omega^\delta$, we obtain 
\begin{align}\label{numerator1}
 & (\tau(\Omega^\delta)-1)\mathcal{H}^{n-1}(\partial^*\Omega^\delta\cap \partial^*(E_{\varphi,\vartheta}\cap \Omega^\delta))\le P(E_{\varphi,\vartheta}\cap \Omega^\delta)-\mathcal{H}^{n-1}(\partial^*\Omega^\delta\cap \partial^*(E_{\varphi,\vartheta}\cap \Omega^\delta)) \nonumber\\
 &\qquad\qquad\qquad\qquad\qquad=  P(E_{\varphi,\vartheta}\cap\Omega^\delta;(\Omega^\delta)^{(1)})
      \le  P(E_{\varphi,\vartheta};(\Omega^\delta)^{(1)})\le P(E_{\varphi,\vartheta};\mathbb B^n).
\end{align}
This gives us the desired upper bound. 

As a result, recalling the form of $\mu$, we obtain from \eqref{denomin} and \eqref{numerator1} that,   whenever $\delta\in (0, \epsilon_0)$,
\begin{align}
    \tau(\Omega^\delta)-1&\le \frac{P(E_{\varphi,\vartheta};\mathbb B^n)}{\mathcal{H}^{n-1}(\partial^*\Omega^\delta\cap \partial^*(E_{\varphi,\vartheta}\cap \Omega^\delta))}\le  \frac{P(E_{\varphi,\vartheta};\mathbb B^n)}{\mathcal{H}^{n-1}(\mathbb S^{n-1}\cap \partial^*E_{\varphi,\vartheta})+ \mathcal{H}^{n-1}(\partial B_{\hat x,k_0}^\delta)} \nonumber\\
  & = P(E_{\varphi,\vartheta};\mathbb B^n) \left[\frac{1}{\mathcal{H}^{n-1}(\mathbb S^{n-1}\cap \partial^*E_{\varphi,\vartheta})}\right.\nonumber\\
   &\quad  \qquad \qquad \left.-\frac{ \mathcal{H}^{n-1}(\partial B_{\hat x,k_0}^\delta)}{(\mathcal{H}^{n-1}(\mathbb S^{n-1}\cap \partial^*E_{\varphi,\vartheta})+ \mathcal{H}^{n-1}(\partial B_{\hat x,k_0}^\delta))\mathcal{H}^{n-1}(\mathbb S^{n-1}\cap \partial^*E_{\varphi,\vartheta})}\right], \label{tau omega 0} 
\end{align}
   where we applied
$$\frac 1{\alpha+\beta}= \frac 1{\beta} - \frac{\alpha}{(\az+\bz)\bz} \quad \text{ with } \ \az=\mathcal{H}^{n-1}(\partial B_{\hat x,k_0}^\delta),\, \beta= \mathcal{H}^{n-1}(\mathbb S^{n-1}\cap \partial^*E_{\varphi,\vartheta})$$
in the last line. Then since \eqref{mu} yields 
$$\mathcal{H}^{n-1}(\partial B_{\hat x,k_0}^\delta)= n\omega_n 2^{-(n-1)k_0}(k_0!)^{-(n-1)/\delta}\ge c(n)\mu(\delta),$$
we continue from \eqref{tau omega 0} to obtain the existence of a constant $c_0=c_0(n)>0$ that
\begin{equation}
\label{tau omega}
   \tau(\Omega^\delta)-1  \le  \frac{P(E_{\varphi,\vartheta};\mathbb B^n)}{\mathcal{H}^{n-1}(\mathbb S^{n-1}\cap \partial^*E_{\varphi,\vartheta})}- c_0 \mu(\delta),
\end{equation}

Furthermore, by Theorem \ref{cite cianchi} and applying Lemma \ref{Maggi} to $E_{\varphi,\vartheta}$, it follows that
$$\tau(\mathbb B^n)=1+P(E_{\varphi,\vartheta};\mathbb B^n)/\mathcal{H}^{n-1}(\mathbb S^{n-1}\cap \partial^*E_{\varphi,\vartheta}).$$ Thus, \eqref{tau omega} yields the first inequality in \eqref{tau o est}.
Therefore,  
the proof of Proposition~\ref{tau value} is complete.
\end{proof}

\begin{proof}[Proof of Theorem~\ref{counexample}]
 For any $\epsilon>0$, using Proposition~\ref{tau value} and  recalling the form of $\mu$ given in \eqref{mu},  there exists $\delta\in (0,\epsilon_0)$ such that $0<\mu(\delta)=C_3(n) 
 (k_0!)^{-(n-1)/\delta}\leq \epsilon /2$. Therefore, taking $D=D^\delta$ and $\Omega=\Omega^\delta$, it follows from \eqref{tau o est} in Proposition~\ref{tau value} that $\tau(\mathbb B^n )>\tau(\Omega)>\tau(\mathbb B^n)-\epsilon$. Moreover, by \eqref{inn bound lem} in Lemma~\ref{ch3 fin per}, we have $$P(\Omega\Delta \mathbb B^n)\le C_1(n)  (k_0!)^{-(n-1)/\delta} \leq C(n)\epsilon,$$
 where $C(n)=C_1(n)/(2C_3(n))$. Additionally, the fact that $D$ consists of countably many balls is a direct consequence of the construction.
\end{proof}

\section{Proof of Theorem \ref{equiva}.}\label{proof eq}
In this section, it is always assumed that $\Omega\subset \mathbb R^n$ satisfies \eqref{almost theoretic}. Our aim is to prove Theorem~\ref{equiva} and Corollary~\ref{iff}. 

\subsection{Basic properties of sets}

We first show the following lemma  stating that, up to a modification on a  set of $\mathcal{L}^n$-measurable zero, $\Omega$ is an open set.
\begin{lem}\label{open rep}
     Given $\Omega\subset\mathbb R^n$ satisfying \eqref{almost theoretic}, define  
    \begin{equation}\label{repres}
        \widetilde \Omega:=\left\{x\in \mathbb R^n: \text{ there exists }r>0 \text{ such that } |\Omega\cap B(x,r)|=|B(x,r)| \right\}.
    \end{equation}
  Then the following properties hold:
  \begin{enumerate}
      \item [(i)]   $\widetilde \Omega$ is open and satisfies 
  \begin{equation}\label{open replace}
    |\widetilde  \Omega\Delta \Omega|=0,\quad \mathcal H^{n-1}(\partial^*\widetilde  \Omega\Delta \partial^* \Omega)=0\quad\text{and}\quad \tau(\widetilde \Omega)=\tau(\Omega).
\end{equation}
Thus,
\begin{equation}\label{new 1.3}
    \partial\widetilde \Omega\subset \partial \Omega\quad  \text{and}\quad  \mathcal{H}^{n-1}(\partial \widetilde  \Omega\setminus\partial^*\widetilde \Omega)=0.
\end{equation}
\item[(ii)] For any $u \in BV(\mathbb R^n)\cap L^{\infty}(\mathbb R^n)$,
\begin{equation}\label{varia replace}
    \|Du\|(\Omega^{(1)})=\|Du\|(\widetilde\Omega).
\end{equation}

  \end{enumerate}

\end{lem}
 
 \begin{proof}[Proof of Lemma 4.1]
 First, for any $x\in \widetilde \Omega$ satisfying $|\Omega\cap B(x,r)|=|B(x,r)|$ with some $r>0$, each point $y\in B(x,r)$ satisfies $B(y,r-|x-y|)\subset B(x,r)$. Therefore,
 $$|B(y,r-|x-y|)\setminus \Omega|\le |B(x,r)\setminus \Omega|=0.$$
 Thus, $y\in \widetilde \Omega$, and by the arbitrariness of $y$ and $x$, we conclude that $\widetilde \Omega$ is open. 

 We next prove \eqref{open replace}. For fixed $x\in \widetilde \Omega$, suppose that $|B(x, r)\cap\Omega|=|B(x, r)|$ for some $r>0$. Then 
 \begin{equation}\label{r' small}
    |B(x,\,r')\cap \Omega|=|B(x,\,r')| \quad \text{ for any } \ 0<r'<r.
 \end{equation}
 Hence, $x$ has fully density, i.e.\ 
 \begin{equation}\label{tilde in 1}
 \widetilde{\Omega} \subset \Omega^{(1)}.    
 \end{equation}
 Since $\Omega^{(1)} \cap \partial^*\Omega = \emptyset$, we have
$$\Omega^{(1)}\Delta\widetilde \Omega=\Omega^{(1)}\setminus\widetilde\Omega\subset \partial\Omega\setminus\partial^*\Omega.$$
Combining this with the assumption that $\mathcal{H}^{n-1}(\partial\Omega \setminus \partial^*\Omega) = 0$, we obtain
\begin{equation}\label{bound zero}
\mathcal{H}^{n-1}(\Omega^{(1)} \Delta \widetilde{\Omega}) \le \mathcal{H}^{n-1}(\partial\Omega \setminus \partial^*\Omega) = 0.
\end{equation}
Furthermore, by the Lebesgue-Besicovitch differentiation theorem (cf. \cite[Theorem 1.34]{EG92}), we have $|\Omega \Delta \Omega^{(1)}| = 0$. Therefore, applying the triangle inequality to $|\Omega \Delta \widetilde{\Omega}|$ and using \eqref{bound zero}, we conclude that
$$|\Omega\Delta\widetilde \Omega|\le |\Omega\Delta \Omega^{(1)}|+|\Omega^{(1)}\Delta\widetilde  \Omega|=0.$$
This immediately implies $\mathcal H^{n-1}(\partial^*\widetilde \Omega\Delta \partial^* \Omega)=0$. Recalling the definitions of $\tau(\widetilde \Omega)$ and $\tau(\Omega)$, we conclude that $\tau(\widetilde \Omega)=\tau(\Omega)$, and thus \eqref{open replace} holds.

   
Next, recalling the construction of $\widetilde \Omega$, we get $
\inter(\Omega)\subset \widetilde \Omega$. Moreover, \eqref{r' small} implies that, each point of $\widetilde\Omega$ is a limit point of $\Omega$; so is each point of $cl(\widetilde\Omega)$, i.e. $cl(\widetilde\Omega)\subset \overline{\Omega}$.  Consequently, $\partial \widetilde{\Omega} \subset \partial \Omega$. This, together with the second property in \eqref{open replace} and the assumption $\mathcal{H}^{n-1}(\partial\Omega \setminus \partial^*\Omega) = 0$, implies $\mathcal{H}^{n-1}(\partial\widetilde \Omega\setminus \partial^*\widetilde \Omega)=0$, thus giving \eqref{new 1.3}.

It remains to show \eqref{varia replace}. Indeed, for any $u\in BV(\mathbb R^n)$, in virtue of \eqref{fr coarea}  and \eqref{bound zero}, we have 
\begin{align*}
    \|Du\|(\Omega^{(1)})&=\int_{\mathbb R}\mathcal{H}^{n-1}(\partial^*\{u>t\}\cap\Omega^{(1)})\\
    & = \int_{\mathbb R}\left(\mathcal{H}^{n-1}(\partial^*\{u>t\}\cap\widetilde\Omega)+\mathcal{H}^{n-1}(\partial^*\{u>t\}\cap(\Omega^{(1)}\setminus\widetilde\Omega))\right)\,dt \nonumber\\
    & \le \|Du\|(\widetilde\Omega)+\int_{\mathbb R}\mathcal{H}^{n-1}(\Omega^{(1)}\Delta\widetilde\Omega)dt=\|Du\|(\widetilde\Omega).\nonumber
\end{align*}
Then recalling \eqref{tilde in 1}, we obtain the reversed inequality, and thus conclude \eqref{varia replace}.
 \end{proof}

Note that admissible domains are bounded by its definition, while such a condition is not imposed in Maz'ya's condition \eqref{admis 2 rep}. We next show that \eqref{admis 2 rep} necessarily gives the boundedness of $\Omega$.

\begin{lem}\label{admis bounded}
  If $\Omega$ is an open set satisfying Maz'ya's condition \eqref{admis 2 rep}, then  $\Omega$ is both connected and bounded.
\end{lem}

In order to prove Lemma \ref{admis bounded},   we introduce the following definition; see \cite[Section 5.5]{F2015}.
\begin{defn}\label{Peri sphere}
    For any $r>0$ and any Borel set $E\subset \partial B_r$, the perimeter of $E$ on the sphere $\partial B_r$ is defined as 
$$P_{\partial B_r}(E)=\sup\left\{\int_{E} {\diver_{\partial B_r}} \phi \, d\mathcal{H}^{n-1}: \phi\in C^{\infty}(\partial B_r;\mathbb R^n),\, \phi(x)\cdot x=0 \text{ for all }x\in\partial B_r, \|\phi\|_{\infty}\le 1\right\},$$
where $\diver_{\partial B_r}$ denotes the tangential divergence on $\partial B_r$.
\end{defn}

Recall that, for any radius $r > 0$ and any pair of points $x, y \in \partial B_r$, the geodesic distance between $x$ and $y$ is given by
$$\dist_{\partial B_r}(x,y):=r\arccos\frac{\langle x,y\rangle}{r^2},$$
where $\langle x, y \rangle$ denotes the Euclidean inner product of $x$ and $y$.
The open geodesic ball with the radius $\rho\in (0, \pi r)$ and the center $q\in \partial B_r$ is denoted by 
$$B_{\partial B_r}(q,\rho):=\{x\in\partial B_r: \dist_{\partial B_r}(x,q)<\rho\}.$$

The following lemma tells the isoperimetric inequality over $\partial B_r$; see \cite{S1943}.

\begin{lem}\label{sphere isoperi}
    Let $r>0$. For a Borel set $E\subset\partial B_r$ with $P_{\partial B_r}(E)<+\infty$, suppose that the geodesic ball $B_{\partial B_r}(q,\rho)$ satisfies $q\in \partial B_r, \,\rho\in (0,\pi r)$ and $\mathcal{H}^{n-1}(E)=\mathcal{H}^{n-1}(B_{\partial B_r}(q,\rho))$. Then 
    $$P_{\partial B_r}(B_{\partial B_r}(q,\rho))\le P_{\partial B_r}(E),$$
     where $P_{\partial B_r}(B_{\partial B_r}(q,\rho))=(n-1)\omega_{n-1}r^{n-2}(\sin(\rho/r))^{n-2}$.
\end{lem}

\begin{lem}\label{sphere slice}
    Let $E\subset \mathbb R^n$ be a Borel set of finite perimeter. Then for a.e. $t\in (0,+\infty)$, 
    the spherical slice $E\cap \partial B_t$ is a Borel set of finite perimeter on $\partial B_t$, and 
    \begin{equation}\label{sphere slice form}
        \int_r^{+\infty} P_{\partial B_t}(E\cap \partial B_t) dt=\int_{\partial^*E\setminus B_r}\sqrt{1-\left(\frac{x}{|x|}\cdot \nu_E\right)^2}\, d\mathcal H^{n-1}\le P(E; (B_r)^c).
    \end{equation}
\end{lem}
\begin{proof}
    This is obtained directly by applying the coarea formula \cite[Theorem 18.8]{M2012} to the radial function $f(x)=|x|$.
\end{proof}

\begin{proof}[Proof of Lemma \ref{admis bounded}.] 
\noindent{\bf Step 1:  $\Omega$ is connected.} First of all, we prove the connectedness of $\Omega$ by contradiction. Suppose that, on the contrary, $\Omega$ is disconnected. Then, by the openness of $\Omega$, there exist two disjoint nonempty open sets $\Omega_1$ and $\Omega_2$ such that $\Omega = \Omega_1 \cup \Omega_2$,  $\partial\Omega_1 \cap \Omega_2 = \emptyset$ and $\partial\Omega_2\cap \Omega_1=\emptyset$, hence    $P(\Omega_1;\Omega)= P(\Omega_2;\Omega) = 0$. Without loss of generality, we may assume that $$0 < P(\Omega_1) = P(\Omega_1;\Omega^c) \le P(\Omega_2;\Omega^c).$$ 
Then, by taking $F = \Omega_1$ in \eqref{admis 2 rep}, we obtain a contradiction. Thus, $\Omega$ is connected.  

\medskip

\noindent{\bf Step 2: $\Omega$ is bounded.}
On the contrary, suppose that $\Omega$ is unbounded. Thanks to the openness of $\Omega$, we have that  
\begin{equation}\label{strictly positive}
P(\Omega;\,(B_r)^c)> 0\ \text{ and } \ |\Omega\cap (B_r)^c|>0 \quad \text{ for every } r>0.
\end{equation} 
Therefore, 
$$|\Omega\cup B_r|=|B_r\cup(\Omega\cap (B_r)^c)|>|B_r| \quad \text{for every } \ r>0.$$
Moreover, $|\Omega\cup B_r|\leq |\Omega|+|B_r|<\infty$. Then the isoperimetric inequality \eqref{iso 2} implies
\begin{equation}\label{iso key}
       P(\Omega\cup B_r)>P(B_r)\quad \text{for every } r>0. 
\end{equation}
We claim that, there exists a constant $\hat R>0$  such that, for a.e. $r>\hat R$, 
\begin{equation}\label{iso control}
   \mathcal{H}^{n-1}(\overline{\Omega}\cap \partial B_r) < P(\Omega; (B_r)^c) < 2C_M\mathcal{H}^{n-1}(\Omega\cap \partial B_r),
\end{equation}
where $C_M>0$ is the constant in \eqref{admis 2 rep}.

\medskip

\noindent{\bf Step 2.1: Verify the first inequality in \eqref{iso control}.}
By \eqref{iso key} and applying Lemma \ref{Maggi} to $\Omega\cup B_r$, we get 
\begin{align}\label{chai}
     & P(\Omega; B_r^{(0)})+P(B_r;\Omega^{(0)})+\mathcal{H}^{n-1}(\{\nu_\Omega=\nu_{B_r}\})\nonumber\\
    &\qquad \qquad \qquad\quad  \qquad=P(\Omega\cup B_r) >  P(B_r)= P(B_r;\overline{\Omega})+P(B_r;(\overline{\Omega})^c)
\end{align}
 for every $r> 0$. Observe that $(\overline{\Omega})^c\subset\Omega^{(0)}$ and $\Omega^{(0)}\setminus (\overline{\Omega})^c\subset \partial\Omega \setminus \partial^*\Omega$. This, together with \eqref{almost theoretic}, implies 
 $$P(B_r;\Omega^{(0)})=P(B_r;(\overline{\Omega})^c).$$
Plugging this equality into \eqref{chai} yields 
\begin{equation}\label{iso cont3}
    P(\Omega; B_r^{(0)})+\mathcal{H}^{n-1}(\{\nu_\Omega=\nu_{B_r}\})>  P(B_r;\overline{\Omega})=\mathcal H^{n-1}(\overline{\Omega}\cap \partial B_r).
\end{equation}
 Furthermore, due to $|\partial\Omega|=0$ assumed in \eqref{almost theoretic}, the coarea formula applied to $|\partial\Omega|$ yields 
\begin{equation}\label{coare}
\mathcal{H}^{n-1}(\partial\Omega\cap \partial B_r)=0\quad  \text{for a.e. }  r\in (0, +\infty). 
\end{equation}
Whence, $\mathcal H^{n-1}(\{\nu_\Omega=\nu_{B_r}\})$ vanishes for almost every $r\in (0, +\infty)$, and substituting this into \eqref{iso cont3} yields the first inequality in \eqref{iso control}.

\medskip
\noindent{\bf Step 2.2: Verify the second inequality in \eqref{iso control}.}  
Since $\Omega$ has finite perimeter, 
$$P(\Omega; B_R)=\|D\chi_\Omega\|(B_R)$$
with $D\chi_\Omega$ a Radon measure. Plugging $r=R$ into the first inequality in \eqref{iso control} implies that,
for any $\epsilon>0$, there exists $R_\epsilon\gg 1$ such that, for any $R>R_\epsilon$, 
\begin{equation}\label{radon}
    P(\Omega)> P(\Omega)-P(\Omega;(B_R)^c)=P(\Omega;B_R)\ge P(\Omega)-\epsilon.
\end{equation}
  This, combined with \eqref{almost theoretic}, yields that there exists  another constant $\hat R>0$ satisfying that 
  $$P(B_{r})\ge P(B_{\hat R})>P(\Omega)\quad \text{and}\quad P(\Omega;B_{r})\ge P(\Omega;B_{\hat R})>\frac{3}{4}P(\Omega) \quad \text{for any }r\ge \hat R.$$
Therefore, for any $r\ge \hat{R}$, 
 \begin{align}\label{separate}
     P(\Omega; (B_{r})^c)&=P(\Omega)-P(\Omega; B_r) <\frac{1}{4}P(\Omega)\nonumber\\
     &=\frac{1}{4}\min\{P(\Omega), P(B_{r})\}< \frac{1}{2}\min\{P(\Omega; B_{r}), P(B_{r})\}.
 \end{align}

Due to the openness of $\Omega$, $\partial^*\Omega \subset \Omega^c$. Hence, 
\begin{equation}\label{inner convert}
    P(\Omega; (B_r)^c)=P(\Omega; (B_r)^c\cap \Omega^c)\quad \text{and}\quad P(\Omega; B_r)=P(\Omega; B_r^{(1)})=P(\Omega; B_r^{(1)}\cap \Omega^c)
\end{equation}
As a result, 
by further using \eqref{coare} and applying Lemma \ref{Maggi} to $\Omega\setminus B_r$, we arrive at
\begin{align}\label{equal 1}
    P(\Omega; (B_r)^c) & =P(\Omega; (B_r)^c\cap \Omega^c)=P(\Omega;B_r^{(0)}\cap \Omega^c)+\mathcal H^{n-1}(\{\nu_\Omega=-\nu_{B_r}\}\cap \Omega^c)\nonumber\\
    &=P(\Omega\setminus B_r;\Omega^c)-P(B_r;\Omega^{(1)}\cap \Omega^c)\quad \text{for a.e. }r\ge \hat R.
\end{align}
Analogously, by \eqref{coare} and applying Lemma \ref{Maggi} to $\Omega\cap B_r$, \eqref{inner convert} yields
\begin{align}\label{equal 2}
      P(\Omega; B_{r}) &=P(\Omega; B_r^{(1)}\cap \Omega^c)=P(\Omega; B_r^{(1)}\cap \Omega^c)+\mathcal H^{n-1}(\{\nu_\Omega=\nu_{B_r}\}\cap \Omega^c)\nonumber\\
      &= P(\Omega\cap B_r;\Omega^c)-P(B_r;\Omega^{(1)}\cap \Omega^c)\quad \text{for a.e. }r\ge \hat R.
  \end{align}
Therefore, due to \eqref{almost theoretic} and the fact that $\Omega^{(1)}\cap \Omega^c\subset \partial\Omega \setminus \partial^*\Omega$, \eqref{equal 1} and \eqref{equal 2} yield, for a.e. $r\ge \hat R$,
\begin{equation}\label{equal}
    0= P(B_r;\Omega^{(1)}\cap \Omega^c)=P(\Omega\setminus B_r; \Omega^c)-P(\Omega; (B_r)^c)=P(\Omega\cap B_r;\Omega^c)-P(\Omega; B_{r}),
 \end{equation}
which, together with \eqref{separate}, implies $P(\Omega\setminus B_r;\Omega^c)<P(\Omega\cap B_r;\Omega^c)$ for a.e. $r\ge \hat R$.
Thus, by plugging $F=\Omega\cap B_r$ into \eqref{admis 2 rep}, we conclude from \eqref{equal} that, for a.e. $r\ge \hat R$,
\begin{equation}\label{use-mazaya}
    P(\Omega; (B_r)^c)=P(\Omega\setminus B_r;\Omega^c)=\min\{P(\Omega\cap B_r;\Omega^c),\,P(\Omega\setminus B_r;\Omega^c)\} \le C_M P(B_r\cap \Omega; \Omega).
\end{equation}

Recall that \eqref{coare} gives $\mathcal H^{n-1}(\{\nu_\Omega=\nu_{B_r}\})=0$ for almost every $r>\hat R$. Then, Lemma \ref{Maggi} applied to $\Omega\cap B_r$, together with the openness of $\Omega,$ gives 
\begin{equation}\label{rhs}
    P(B_r\cap \Omega; \Omega)= P(B_r; \Omega^{(1)}\cap \Omega)+P(\Omega; B_r^{(1)}\cap \Omega)+\mathcal H^{n-1}(\{\nu_\Omega=\nu_{B_r}\}\cap \Omega)=P(B_r; \Omega^{(1)}\cap \Omega).
\end{equation}
Thanks to $\Omega\subset\Omega^{(1)}$ and $\mathcal{H}^{n-1}(\Omega\cap \partial B_r)=P(B_r; \Omega)$, \eqref{use-mazaya} and \eqref{rhs} yield
$$P(\Omega; (B_r)^c)\le C_M\mathcal{H}^{n-1}(\Omega\cap \partial B_r).$$
This together with \eqref{strictly positive} particularly yields,   
$$\mathcal{H}^{n-1}(\Omega\cap \partial B_r)>0\quad \text{ for a.e. $r\ge \hat R$,}$$
and gives the second inequality in \eqref{iso control}.

\medskip
\noindent{\bf Step 2.3: The decay estimate of $P(\Omega;(B_r)^c)$.}
Set
$$
    f(r):=\int_{r}^{+\infty} P_{\partial B_t}(\Omega\cap \partial B_t)\,dt\geq 0\quad \text{for any } \ r\in [\hat R, +\infty),
$$
which is well-defined  and absolutely continuous with respect to $r$  due to Lemma~\ref{sphere slice}. Then 
\begin{equation}\label{deriv}
    f'(r)=
-P_{ \partial B_r}(\Omega\cap \partial B_r) \quad  \text{ for a.e. }  \ r \in [ \hat R,+\infty).
\end{equation}
Observe that by \eqref{iso control} and applying  Lemma \ref{sphere slice} to the open set $\Omega$,  for a.e. $r\ge \hat R$, $\Omega\cap \partial B_r$ is a Borel set satisfying 
 \begin{align}\label{contra func 2}
 0\leq f(r)\le P(\Omega; (B_r)^c)< 2C_M \mathcal{H}^{n-1}(\Omega\cap \partial B_r).
 \end{align}

Now, by Lemma \ref{sphere isoperi} and \eqref{deriv}, for a.e. $r\ge \hat R$, there exists $B_{\partial B_r}(q_r, \rho_r)$ for some $ q_r\in \partial B_r$ with
 $\mathcal{H}^{n-1}(B_{\partial B_r}(q_r,\rho_r))=\mathcal{H}^{n-1}(\Omega\cap \partial B_r)$, such that
 \begin{equation}\label{sph isoper}
     -f'(r)\ge  P_{\partial B_r}(B_{\partial B_r}(q_r,\rho_r))\quad \text{ for a.e. }r\in [\hat R,+\infty).
 \end{equation}

By gathering \eqref{separate} and \eqref{iso control}, we have
$$
   \mathcal{H}^{n-1}(B_{\partial B_r}(q_r,\rho_r))=\mathcal{H}^{n-1}(\Omega\cap \partial B_r)\le  \mathcal{H}^{n-1}(\overline{\Omega}\cap \partial B_r)< P(\Omega; (B_r)^c)< \frac{1}{2} P(B_r),
$$
which yields $\rho_r/r< \pi/2$. Thus, $B_{\partial B_r}(q_r,\rho_r)$ is contained in some hemisphere of $\partial B_r$, and hence
 \begin{equation}\label{radius control}
 \mathcal{H}^{n-1}(B_{\partial B_r}(q_r,\rho_r))\le C(n)\rho_r^{n-1} \quad  \text{and} \quad P_{ \partial B_r}(B_{\partial B_r}(q_r,\rho_r))\ge C(n)\rho_r^{n-2}.
 \end{equation}
  Hence, via \eqref{sph isoper} and \eqref{radius control}, we deduce that for a.e. $r\in [\hat R,+\infty)$,
$$
    -f'(r)\ge C(n)\rho_r^{n-2}\ge C(n) \left(\mathcal{H}^{n-1}(B_{\partial B_r}(q_r,\rho_r))\right)^{\frac{n-2}{n-1}}= C(n)\left(\mathcal{H}^{n-1}(\Omega\cap \partial B_r)\right)^{\frac{n-2}{n-1}}.
$$
\medskip
\noindent{\bf Step 2.4: Leading to a contradiction.}
Recalling \eqref{contra func 2}, we obtain
\begin{equation}\label{f rela}
    -f'(r)> C(n,C_M) (f(r))^{\frac{n-2}{n-1}}\ge 0 \quad \text{ for a.e. } r\in [\hat R,+\infty).
\end{equation}
Moreover, it follows from \eqref{radon} that 
$$
    \lim_{r\rightarrow \infty} f(r)=0,
$$
which, together with \eqref{f rela}, implies $f(r)>0$ for all $r\in [\hat R,+\infty)$. By integrating $\frac{d}{dr}\left(f^{\frac{1}{n-1}}\right)$ and applying  \eqref{f rela} again,
\begin{equation}\label{contra point}
    f(h)^{\frac{1}{n-1}}\ge f(r)^{\frac{1}{n-1}}+(r-h)C(n,C_M)\quad \text{ for any } r>h\geq \hat R.
\end{equation}
By letting $r \to +\infty$, we conclude that $f(h) = +\infty$, which leads to a contradiction. Hence, $\Omega$ must be bounded.
\end{proof}

\subsection{Equivalence of three descriptions.}\label{equiva proof}

\begin{proof}[Proof of Theorem \ref{equiva}]
    \noindent{\bf Step 1: (i) $\Rightarrow$
 (ii).}  Assume that $\tau(\Omega)\geq 1+\delta_0$ for some $\delta_0>0$. By Lemma \ref{open rep}, up to a modification of $\mathcal L^n$-measure zero set, $\Omega$ is an open set that satisfies the condition \eqref{almost theoretic}. 
Therefore, to obtain (ii), it suffices to verify \eqref{admis 2 rep} for $\Omega$ according to Lemma \ref{mazya s and n}.  

Take an arbitrary  $\mathcal L^n$-measurable set $F\subset \Omega$. Without loss of generality, we may assume that $P(F; \Omega)<\infty$. Then it follows from \cite[Lemma 1, Page 489]{M2011} and \eqref{almost theoretic} that $P(F)<\infty$.  If $|F|=0$ (resp. $|\Omega\setminus F|=0$), then $P(F; \Omega^c)=0$  (resp. $P(\Omega\setminus F;\Omega^c)=0$), and hence \eqref{admis 2 rep} holds automatically.

If $0<|F|\le |\Omega|/2$, by Lemma \ref{Maggi} and the definition of $\tau(\Omega)$, it follows that 
\begin{equation}\label{F-small}
\tau(\Omega)P(F;\Omega^c)=\tau(\Omega)\mathcal{H}^{n-1}(\partial^*\Omega\cap \partial^* F)\le P(F)=P(F;\Omega)+P(F;\Omega^c).
\end{equation}
Thus, \eqref{admis 2 rep} holds with $C_M=(\tau(\Omega)-1)^{-1}\leq \delta^{-1}_0$. 

If $|\Omega|/2<|F|< |\Omega|$, then $0<|\Omega\setminus F|\le |\Omega|/2$. Since $\{\nu_\Omega=\nu_F\}\cap \Omega=\emptyset$ and $\Omega
\subset\Omega^{(1)}$ via the openness of $\Omega$,
by Lemma \ref{Maggi} and the fact that $\Omega\cap F^{(0)}\subset \Omega$, we have
$$
    P(\Omega\setminus F;\Omega)=P(F;\Omega)+P(\Omega;F^{(0)}\cap\Omega)+\mathcal{H}^{n-1}(\{\nu_\Omega=-\nu_F\}\cap \Omega)=P(F;\Omega). 
$$
Therefore, repeating the argument in \eqref{F-small} by replacing $F$ with $\Omega\setminus F$, we conclude that
$$P(\Omega\setminus F; \Omega^c)\leq (\tau(\Omega)-1)^{-1} P(\Omega\setminus F; \Omega)=(\tau(\Omega)-1)^{-1} P(F; \Omega),$$
which gives \eqref{admis 2 rep} by taking $C_M=(\tau(\Omega)-1)^{-1}\leq \delta^{-1}_0$.  Hence, \eqref{admis 2 rep} holds for any $\mathcal L^n$-measurable set $F\subset \Omega$. This completes the proof of (i) $\Rightarrow$ (ii).
\smallskip

\noindent{\bf Step 2: (ii) $\Rightarrow$ (iii).} Suppose that $\Omega$ is an open set satisfying \eqref{11trace} with $C_T>0$.
By Lemma \ref{mazya s and n}, $\Omega$ satisfies Maz'ya's condition \eqref{admis 2 rep} with $C_M=C_T$, and then Lemma \ref{admis bounded} shows that $\Omega$ is a bounded domain. 

Next, we show that there exists a positive constant $\Theta=\Theta(\Omega)>0$ such that, for each $x\in\partial\Omega$, there is a ball $B(x,r)$ such that, for any set $E\subset\overline{\Omega}\cap B(x,r)$ of finite perimeter, 
\begin{equation}\label{admissible-later}
	\mathcal H^{n-1}(\partial^{*} \Omega\cap \partial^* E)\le \Theta \mathcal H^{n-1}(\partial^* E\cap \Omega).
\end{equation}
Toward this, recall that  $\Omega$ is an open set satisfying \eqref{almost theoretic}. Then $|\partial \Omega|=0$ and, up to replacing $E$ by $E\cap \Omega$ if necessary, we may assume $E\subset \Omega \cap B(x,r)$.

Fix $x\in \partial\Omega$. Observe that $\mathcal H^{n-1}$ is Borel regular, and that $\mathcal{H}^{n-1}(\partial\Omega)<\infty$ follows from \eqref{almost theoretic} and $$\mathcal{H}^{n-1}(\partial^*\Omega)=P(\Omega)<+\infty.$$
Thus, by \cite[Theorem 1.7]{EG92}, $\mathcal{H}^{n-1}\llcorner \partial\Omega$ is a Radon measure. Hence, applying \cite[Theorem 1.8]{EG92} to $\mathcal H^{n-1}\llcorner\partial\Omega$  
permits us to
choose $r>0$ small enough so that 
\begin{equation}\label{small t judge}
	\mathcal H^{n-1}(\partial\Omega\cap B(x,2r))<\frac{1}{3}P(\Omega).
\end{equation}
Via Lemma \ref{Maggi}, for any set of finite perimeter $E\subset\Omega\cap B(x,r)$, we have
\begin{equation}\label{P-H}
	P(E; \Omega^c)=	\mathcal H^{n-1}(\partial^{*} \Omega\cap \partial^* E) \quad \text{and}\quad P(E; \Omega)=\mathcal H^{n-1}(\partial^* E\cap \Omega).
	\end{equation}
On the one hand, since $E\subset B(x,r)$,
 it follows from \eqref{small t judge} and \eqref{P-H} that
\begin{equation}\label{small term}
P(E;\Omega^c)=\mathcal{H}^{n-1}(\partial^*\Omega\cap \partial^* E)\le \mathcal H^{n-1}(\partial^*\Omega\cap B(x,2r))<\frac{1}{3}P(\Omega).
\end{equation}
On the other hand, as the openness of $\Omega$ gives  $\partial^*\Omega\subset \Omega^c$, Lemma \ref{Maggi} tells that 
\begin{equation}\label{divide 3}
    P(\Omega\setminus E;\Omega^c)\ge P(\Omega;E^{(0)}\cap \Omega^c)=P(\Omega;E^{(0)}),\nonumber
\end{equation}
which, together with $B(x,2r)^c\subset E^{(0)}$ and \eqref{small t judge}, implies
\begin{equation}\label{big term}
    P(\Omega\setminus E;\Omega^c)\ge P(\Omega;B(x,2r)^c)= P(\Omega)- \mathcal H^{n-1}(\partial^*\Omega\cap B(x,2r))\ge \frac{2}{3}P(\Omega).
\end{equation}

As a result, according to \eqref{small term} and \eqref{big term}, Maz'ya's condition \eqref{admis 2 rep} implies
$$P(E; \Omega^c)=\min\{P(E; \Omega^c), P(\Omega\setminus E; \Omega^c)\}\le C_M P(E;\Omega).$$
Hence, recalling \eqref{P-H}, we conclude that \eqref{admissible-later} holds with $\Theta:=C_M=C_T$. Therefore, $\Omega$ is an $\Theta$-admissible domain, finishing the proof of (ii) $\Rightarrow $ (iii). 
\smallskip

\noindent{\bf Step 3: (iii) $\Rightarrow$ (i).} Suppose that $\Omega$ is an $\Theta$-admissible domain. 
Then, it follows from
\cite[Theorem 5.10.7]{Z1989} that there exists a constant $M=M(\Omega)>0$, such that, for any $u\in BV(\Omega)$, the trace $Tu$ of the function $u$ satisfies
\begin{equation}\label{trace int}
    \int_{\partial^*\Omega}|Tu|d\mathcal{H}^{n-1}\le M\|u\|_{BV(\Omega)}.
\end{equation}
Note that by \cite[Lemma 9.6.3]{M2011},  $T\chi_F=\chi_{\partial^*F}$ holds for $\mathcal H^{n-1}$-a.e. $x\in \partial^*\Omega$. Therefore, for any set $F\subset\Omega$ of finite perimeter with $0<|F|\le |\Omega|/2$, plugging  $\chi_{F}$ into \eqref{trace int} yields
\begin{equation}\label{trace-adm}
\mathcal{H}^{n-1}(\partial^*\Omega\cap \partial^*F )=\int_{\partial^*\Omega}|T\chi_F|d\mathcal{H}^{n-1}\le M\left(\|\chi_F\|_{L^1(\Omega)}+\|D\chi_F\|(\Omega)\right),
\end{equation}

We claim that there is a constant $C'>0$ independent of $F$, such that whenever $F$ is a set of finite perimeter with $0<|F|\leq |\Omega|/2$,
\begin{equation}\label{poin ex}
    \|\chi_F\|_{L^1(\Omega)}\le C'\|D\chi_F\|(\Omega).
\end{equation}
 Once \eqref{poin ex} holds, then it follows from \eqref{trace-adm} that
$$\mathcal{H}^{n-1}(\partial^*\Omega\cap \partial^*F)\leq M(1+C') \|D\chi_F\|(\Omega) = M(1+C') P(F;\Omega).$$
Since $\mathcal{H}^{n-1}(\partial^*\Omega\cap \partial^*F)=P(F; \Omega^c)$ and $P(F)=P(F; \Omega)+P(F; \Omega^c)$, we further obtain that
$$\frac{P(F)}{\mathcal{H}^{n-1}(\partial^*\Omega\cap \partial^*F)}\geq 1+\frac{1}{M(1+C')}$$ 
whenever $F$ is a set of finite perimeter with $0<|F|\leq |\Omega|/2$. Therefore, 
 $\tau(\Omega)\geq  1+\delta_0$ with \\$\delta_0=1/(M(1+C'))>0$.

Now we prove \eqref{poin ex}. Write  
$$(\chi_F)_\Omega=|F|/|\Omega|$$
the integral average of $\chi_F$ over $\Omega$. As $\Omega$ is $\Theta$-admissible, then it supports the Poincar\'e inequality \cite[Theorem 5.11.1]{Z1989}, i.e. 
\begin{equation}\label{since}
    \|\chi_F-(\chi_F)_\Omega\|_{L^1(\Omega)}\le C''\|D\chi_F\|(\Omega),
\end{equation}
where $C''=C''(\Omega,n)>0$. Since $|F|\le |\Omega|/2$,  we  have
$$\|\chi_F\|_{L^1(\Omega)}=|F|\le 2\frac{|\Omega|-|F|}{|\Omega|}|F|=\|\chi_F-(\chi_F)_\Omega\|_{L^1(\Omega)},$$
which, together with \eqref{since}, gives \eqref{poin ex} with $C'=C''$.  This finishes the proof of (iii) $\Rightarrow $ (i).
\end{proof}

\subsection{Proof of Corollary \ref{iff}.}

\begin{proof}[Proof of Corollary \ref{iff}]
We begin with the proof of (i), which is divided into two steps.

\noindent{\bf Step 1: Prove that $\Omega$ is a John domain.} As $\Omega$ satisfies \eqref{11trace} with constant $C_T$, then, by Lemma \ref{mazya s and n}, \eqref{admis 2 rep} holds for the constant $C_M=C_T$.
Hence, by Lemma \ref{admis bounded}, $\Omega$ is a bounded domain. 

Moreover, according to Lemma \ref{Mazya-extension} and \eqref{admis 2 rep} (with $C_M=C_T$), it follows that  for any $u\in BV(\Omega)$, we can find an extension $\hat u\in BV(\mathbb R^n)$ satisfying that $\hat u=u$ over $\Omega$ and that $\hat u= c_0$ over $\Omega^c$ for some constant $c_0\in \mathbb R$,
such that \eqref{ext} holds (with $C_M=C_T$). Thus, by \eqref{ext} and the Sobolev inequality for $BV$-functions (cf. \cite[Theorem 5.10]{EG92}), we have
$$\inf_{c\in\mathbb R}\left(\int_\Omega| u-c|^{1^*}\right)^{1/1^*}\le \left(\int_{\mathbb R^n}| \hat u-c_0|^{1^*}\right)^{1/1^*}\le C(n)\|D\hat u\|(\mathbb R^n)\le C(n,C_T)\|Du\|(\Omega).$$
Therefore, recalling that $\Omega$ is a bounded domain satisfying the ball separation property with respect to $x_0$ and the constant $S>0$, we conclude from Lemma \ref{Sobolev John} that $\Omega$ is a $J$-John domain with John center $x_0$, where 
$J=J(n, C_T, S, |\Omega|/r_0^n)$ and $r_0=\dist(x_0, \Omega^c)$.

\medskip
 \noindent{\bf Step 2: Prove \eqref{hau cond 1}.} Given any $r\in (0,r_0)$ and $x\in \partial\Omega$, in order to get \eqref{hau cond 1}, we first estimate 
 $$ \int_{\partial\Omega}T\chi_{\Omega\cap B(x,r)}d \mathcal H^{n-1}.$$ 
 
 Toward this, let 
 $$I_{x,r}:=\{z\in \partial\Omega: T\chi_{\Omega\cap B(x,r)}(z)=0\}.$$
 Suppose that $\mathcal{H}^{n-1}(I_{x,r})>0$. Then for any constant $c\in\mathbb R$, via the triangle inequality  we obtain 
\begin{align*} &\int_{\partial\Omega}|T\chi_{\Omega\cap B(x,r)}|d\mathcal{H}^{n-1}- \int_{\partial\Omega}|T\chi_{\Omega\cap B(x,r)}-c|d\mathcal{H}^{n-1}\le \int_{\partial\Omega}|c|d\mathcal{H}^{n-1}(z)\nonumber\\
= & \int_{\partial\Omega}\bint_{I_{x,r}} |T\chi_{\Omega\cap B(x,r)}(y)-c| d\mathcal{H}^{n-1}(y)d\mathcal{H}^{n-1}(z) \le \frac{\mathcal H^{n-1}(\partial\Omega)}{\mathcal H^{n-1}(I_{x,r})}\int_{\partial\Omega}|T\chi_{\Omega\cap B(x,r)}(y)-c|d\mathcal{H}^{n-1}(y).
\end{align*}
Thus, by further applying \eqref{11trace} to $\chi_{\Omega\cap B(x,r)}$, we immediately get
\begin{equation}\label{int of trace}
\int_{\partial\Omega}|T\chi_{\Omega\cap B(x,r)}|d\mathcal{H}^{n-1} \le \left(1+\frac{\mathcal H^{n-1}(\partial\Omega)}{\mathcal H^{n-1}(I_{x,r})}\right) C_T \|D\chi_{\Omega\cap B(x,r)}\|(\Omega) \quad \text{ when } \ \mathcal{H}^{n-1}(I_{x,r})>0.
\end{equation}

Separately, since $\Omega$ is open and
Lemma \ref{Maggi} applied to $\partial^*(\Omega\cap B(x,r))$ yields
\begin{equation}\label{cap ball}
    \partial^*(\Omega\cap B(x,r))= (\partial^*\Omega \cap B(x,r))\cup\{\nu_\Omega=\nu_{B(x,r)}\}\cup (\partial B(x,r)\cap \Omega^{(1)}),
\end{equation}
it follows that
$$P(\Omega\cap B(x,r);\Omega)\le P(B(x,r); \overline{\Omega})\le n\omega_n r^{n-1}.$$ 
Therefore,  \eqref{int of trace} implies that, when $\mathcal{H}^{n-1}(I_{x,r})>0$,
\begin{align}\label{appex-1}
    \int_{\partial\Omega}|T\chi_{\Omega\cap B(x,r)}|d \mathcal H^{n-1}& \le \left(1+\frac{\mathcal H^{n-1}(\partial\Omega)}{\mathcal H^{n-1}(I_{x,r})}\right) C_T P(\Omega\cap B(x,r);\Omega)\nonumber\\
    &\le \left(1+\frac{\mathcal H^{n-1}(\partial\Omega)}{\mathcal H^{n-1}(I_{x,r})}\right) C_T n\omega_n r^{n-1} .
\end{align}

Next, we claim that,  for $0<r<r_0=\dist(x_0,\, \Omega^c)$,
\begin{equation}\label{vanish}
	\mathcal H^{n-1}(\partial\Omega\cap B(x,r))\le \int_{\partial\Omega}T\chi_{\Omega\cap B(x,r)}d \mathcal H^{n-1}\quad \text{and}\quad \mathcal{H}^{n-1}(I_{x,r}) \ge \frac{1}{2}n\omega_n r_0^{n-1}.
\end{equation}
Once \eqref{vanish} holds, combining it with \eqref{appex-1}
and, by the assumption \eqref{almost theoretic},
$$\mathcal H^{n-1}(\partial \Omega)=\mathcal H^{n-1}(\partial^* \Omega)=P(\Omega),$$
we obtain \eqref{hau cond 1} for $\Omega$ with $C_H=C_T(n\omega_n+2P(\Omega)/r_0^{n-1})$.

Now we prove \eqref{vanish}. 
Then by \eqref{almost theoretic}, \eqref{cap ball} and applying \cite[Lemma 9.6.3]{M2011} to $\Omega\cap B(x,r)$, we obtain that for $\mathcal H^{n-1}$-a.e. $ z\in \partial\Omega$,
\begin{equation}\label{trace range}
    \chi_{\partial^*\Omega \cap \overline{B}(x,r)}(z)\ge T\chi_{\Omega\cap B(x,r)}(z)=\chi_{\partial^*(\Omega\cap B(x,r))}(z)\ge \chi_{\partial^*\Omega \cap B(x,r)}(z).
\end{equation}
Hence, by recalling \eqref{almost theoretic} and integrating the second inequality in \eqref{trace range} over $\partial \Omega$, we obtain the first inequality in \eqref{vanish}.

It remains to show the second inequality in \eqref{vanish}. Recalling that $x\in\partial\Omega$ and $x_0\in\Omega$, we note that $r< r_0=\dist(x_0,\Omega^c)$ implies $x_0\notin \overline{B}(x,r)$. Thus, there exists an open half-space $H^+\subset \mathbb R^n$, such that  $\overline{B}(x,r)\subset H^+$ and $x_0\in \partial H^+$. Consequently, recalling the definition of  $I_{x,r}$, the first inequality of \eqref{trace range}  together with \eqref{almost theoretic} yields
\begin{align}\label{Ix 1 est}
    \mathcal H^{n-1} (I_{x,r})&\ge \mathcal H^{n-1}(\{z\in \partial\Omega:\chi_{\partial^*\Omega \cap \overline{B}(x,r)}(z)=0\})\nonumber\\
    &\ge \mathcal H^{n-1}(\partial \Omega\setminus \overline{B}(x,r))\ge \mathcal H^{n-1}(\partial \Omega\setminus H^+).
\end{align}

Let $\Omega_s$ be the set which is symmetric with respect to $\partial H^+$ and satisfies $\Omega_s\setminus H^+= \Omega\setminus H^+$. Then since $\Omega\setminus H^+$ is a set of finite perimeter with finite Lebesgue measure, its reflection $\Omega_s\cap \overline{H}^+$ with respect to $\partial H^+$ is also of finite perimeter  with finite Lebesgue measure; so is their union $\Omega_s=(\Omega_s\cap \overline{H}^+)\cup(\Omega\setminus H^+)$. Furthermore, due to the symmetry of $\Omega_s$, each limit point $z\notin H^+$ of $\Omega_s$ is also the limit point of $\Omega\setminus H^+$, thus giving  $\partial\Omega_s\setminus H^+\subset \partial\Omega\setminus H^+$. As a result, via $\partial^*\Omega_s\subset \partial\Omega_s$, we have
\begin{equation}\label{sym peri}
    P(\Omega_s)\le \mathcal{H}^{n-1}(\partial\Omega_s)\le 2\mathcal{H}^{n-1}(\partial\Omega_s\setminus H^+)\le 2\mathcal{H}^{n-1}(\partial\Omega\setminus H^+).
\end{equation}

Moreover, as $x_0\in \partial H^+$ and  $B(x_0,r_0)\subset \Omega$, it also holds that $B(x_0,r_0)\subset\Omega_s$. Therefore, combining \eqref{Ix 1 est}, \eqref{sym peri} and the isoperimetric inequality \eqref{iso 2}, we have 
$$\mathcal{H}^{n-1}(I_{x,r})\ge \mathcal{H}^{n-1}(\partial\Omega\setminus H^+)\ge \frac{1}{2}P(\Omega_s)\ge \frac{1}{2}P(B(x_0,r_0))=\frac{1}{2}n\omega_n r_0^{n-1},$$
which gives the second inequality in \eqref{vanish}.
Hence, we complete the proof of (i).

\medskip
 \noindent{\bf Step 3: Prove (ii).} 
Suppose that $\Omega$ satisfies the assumptions in (ii). Since $\Omega$ is a $J$-John domain with John center $x_0$, it follows from Definition~\ref{John} that, for every point $x\in \Omega$, any John curve $\gamma$ joining $x$ to $x_0$, and each point $y\in \gamma$, one has
$$\gamma[x,\,y]\subset B(y,\,J\dist(y,\,\partial \Omega)).$$
Thus, $\Omega$ satisfies the ball separation property with respect to $x_0$ and the constant $S=J$. 

Next, we establish \eqref{11trace} for $\Omega$. Since the $J$-John domain $\Omega$ satisfies \eqref{hau cond 1} and \eqref{almost theoretic},  \cite[Theorem 1.3]{SZ2024}
tells that for any $u\in W^{1,1}(\Omega)$, the trace $Tu$ of $u$ exists for $\mathcal H^{n-1}$-a.e. $x\in \partial\Omega$.
Furthermore, \cite[Theorem 1.3]{SZ2024} also implies that, for any $u\in W^{1,1}(\Omega)$, $\Omega$ supports the trace inequality \eqref{11trace} with constant $C_T=C_T(n, J, C_H, r_0)>0$, where $C_H>0$ is the same constant in \eqref{hau cond 1}. 

In addition,
by \cite[Theorem 3.4]{LLW21}, for any $u\in BV(\Omega)$, there exists a function $v\in W^{1,1}(\Omega)$ such that $\|Dv\|(\Omega)\leq 2\|Du\|( \Omega)$ and 
	$$\bint_{B(x,r)\cap \Omega}|v-u| \, dy\to 0\quad \text{as }r\to 0^+$$
	uniformly for all $x\in\partial\Omega$. Since the trace function $Tv$ of $v$ exists  for $\mathcal H^{n-1}$-a.e. $x\in \partial\Omega$, the trace function $Tu$ of $u$ exists and $Tu=Tv$ for $\mathcal H^{n-1}$-a.e. $x\in \partial\Omega$.
 Therefore, for any $u\in BV(\Omega)$, we have
$$\inf_{c\in \mathbb R}\int_{\partial\Omega}|Tu-c|d\mathcal{H}^{n-1}=\inf_{c\in \mathbb R}\int_{\partial\Omega}|Tv-c|d\mathcal{H}^{n-1}\le C_T \|Dv\|_{L^1(\Omega)}\leq 2C_T\|Du\|(\Omega),$$
which gives the trace inequality \eqref{11trace}. Hence, the proof of (ii) is completed.
\end{proof}

\end{document}